\documentclass[reqno, 11pt]{article}

\usepackage{fullpage}

\usepackage[centertags]{amsmath}
\usepackage{amsmath}
\usepackage{amsfonts}
\usepackage{amssymb}
\usepackage{amsthm}
\usepackage{enumerate}

\usepackage{newlfont}
\usepackage{color}

\usepackage{authblk}

\usepackage{stmaryrd}
\SetSymbolFont{stmry}{bold}{U}{stmry}{m}{n}

\usepackage{bbm}
\include{amsthm_sc}

\usepackage{stmaryrd}
\usepackage{mathrsfs}

\usepackage{fancyhdr}
\usepackage{graphicx}
\usepackage{fancybox}
\usepackage{setspace}
\usepackage{cleveref}

\newtheorem{thm}{Theorem}
\newtheorem{cor}[thm]{Corollary}
\newtheorem{lem}[thm]{Lemma}

\newtheorem{cond}[thm]{Condition}
\theoremstyle{definition}
\newtheorem{defn}[thm]{Definition}

\theoremstyle{remark}
\newtheorem{rem}[thm]{Remark}

\newcommand{\R} {\mathbb{R}}
\newcommand{\C} {\mathbb{C}}

\newcommand{\E} {\mathbb{E}}

\newcommand{\p} {\mathbb{P}}

\DeclareMathOperator{\diag}{diag}

\DeclareMathOperator{\Tr}{Tr}

\DeclareMathOperator{\supp}{supp}

\DeclareMathOperator{\re}{\mathrm{Re}}
\DeclareMathOperator{\im}{\mathrm{Im}}

\newcommand{\caN}{{\mathcal N}}

\newcommand{\caP}{{\mathcal P}}

\newcommand{\bss}{{\boldsymbol s}}

\newcommand{\bsv}{{\boldsymbol v}}

\newcommand{\bsx}{{\boldsymbol x}}
\newcommand{\bsy}{{\boldsymbol y}}

\newcommand{\wt}{\widetilde}

\newcommand{\beq}{ \begin{equation} }
\newcommand{\eeq}{ \end{equation} }

\newcommand{\lone}{\mathbbm{1}} 

\newcommand{\dd}{\mathrm{d}}
\newcommand{\ii}{\mathrm{i}}

\renewcommand{\bss}{\boldsymbol{\sigma}}

\newcommand{\mean}{M}

\numberwithin{equation}{section} 
\numberwithin{thm}{section}

\title{Fluctuations of the free energy of the spherical Sherrington--Kirkpatrick model} 

\author{Jinho Baik\footnote{Department of Mathematics, University of Michigan,
Ann Arbor, MI, 48109, USA \newline email: \texttt{baik@umich.edu}}
and Ji Oon Lee\footnote{Department of Mathematical Sciences, KAIST, Daejeon, 305701, Korea
\newline email: \texttt{jioon.lee@kaist.edu}}}

\date{\today}

\begin{document}

\maketitle

\begin{abstract}
The spherical Sherrington--Kirkpatrick model is a spherical mean field model for spin glass. 
We consider the fluctuations of the free energy at arbitrary non-critical temperature for the 2-spin model with no magnetic field.  
We show that in the high temperature regime the law of the fluctuations converges to the Gaussian distribution just like in the Sherrington-Kirtkpatrick model. 
We show, on the other hand, that the law of the fluctuations is given by the GOE Tracy-Widom distribution in the low temperature regime.
The orders of the fluctuations are markedly different in these two regimes.  
A universality of the limit law is also proved. 
\end{abstract}

\section{Introduction}

The spherical Sherrington--Kirkpatrick (SSK) model (with $2$-spin interaction and no magnetic field) is defined by the Hamiltonian
\beq \label{SSK Hamiltonian}
	H_N(\bss)= - \frac1{\sqrt{N}} \sum_{i \neq j} J_{ij} \sigma_i \sigma_j 
	= - \frac1{\sqrt{N}} \langle \bss, J \bss\rangle
\eeq
where the spin variables $\bss=(\sigma_1, \cdots, \sigma_N)\in \R^N$ lie on the sphere $\|\bss\|^2=\sum_{i=1}^N \sigma_i^2=N$, 
and $\langle \cdot, \cdot \rangle$ denotes the Euclidean inner product. 
Here $J_{ij}=J_{ji}$, $1\le i<j\le N$, are independent identically distributed random variables, representing the disorder of the system, and $J=(J_{ij})_{i,j=1}^N$ with $J_{ii}=0$. 
We assume that $J_{12}$ has mean $0$ and variance $1$.
The free energy at inverse temperature $\beta$ is defined by
\beq \label{SSK free e}
	F_N= F_N(\beta)= \frac1{N} \log Z_N, 
	\qquad
	Z_N = \int_{S_{N-1}} e^{\beta H_N(\bss)} \dd \omega_N(\bss),
\eeq
where $\dd\omega_N$ is the normalized uniform measure on the sphere $S_{N-1} = \{ \bss \in \R^N : \| \bss \|^2 = N \}$.
Note that since $J_{ij}$ are random, $F_N$ is a random variable. 
The subject of this paper is the fluctuations of $F_N$ as $N\to\infty$. 

The usual Sherrington--Kirkpatrick (SK) model is a mean field version of the Edwards-Anderson model of spin glass, 
and is given by the same Hamiltonian as~\eqref{SSK Hamiltonian} but with the condition that the spin variables are on a lattice instead of the sphere: $\bss \in \{-1, 1\}^N$. 
(For this case many literatures use a different convention that the Hamiltonian is divided by $2$.) 
The partition function is defined as $Z_N= \sum_{\bss \in \{-1, 1\}^N} e^{\beta H_N(\bss)}$. 
Among the numerous existing results on the SK model (see, for example, \cite{TalagrandBook2011v1, TalagrandBook2011v2, PanchenkoBook2013}), 
we here review a few about the free energy that are relevant to this paper. 
The non-random limit of the free energy, $\lim_{N\to \infty} F_N$, for the SK model was first predicted by Parisi \cite{Parisi} in 
the more general setting of $p$-spin interactions in the presence of external magnetic field. 
The Parisi formula was rigorously proved by Talagrand in his famous 2006 paper \cite{TalagrandParisi} in which he proved the convergence of the expectation to the Parisi formula when the disorder random variables $J_{ij}$ are Gaussian. 
The universality of the limit of $F_N$ independent of $J$ is proved under the finite third moment condition, with mean $0$ and variance $1$ by Carmona and Hu \cite{CarmonaHu}, which improved the previous result 
by Guerra and Toninelli \cite{GuerraToninelli2002} for symmetric random variables with finite fourth moment. 
The Parisi formula is implicit and is given in terms of a variational problem. 
For a recent study on this variational problem, see \cite{AuffingerChen}. 
An important feature here is the existence of the critical temperature $\beta_c=\frac12$. 
The cardinality of the support of the measure that underlies in the Parisi formula changes at $\beta=\beta_c$. 
The phase transition is also understood by the fact that the difference between the quenched disorder free energy and the annealed free energy, $\frac1{N} \mathbb{E}[\log Z_N] - \frac1{N} \log \mathbb{E}[Z_N]$, tends to zero 
as $N\to \infty$ in the high temperature regime, $\beta<\beta_c$, but does not tend to zero in the low temperature regime, $\beta>\beta_c$ (see \cite{AizenmanLebowitzRuelle}). 

The fluctuations of the free energy for the SK model in the high temperature regime was studied by Aizenman, Lebowtiz, and Ruelle \cite{AizenmanLebowitzRuelle}. 
They showed that if the disorder random variables $J_{ij}$, $1\le i<j\le N$, are independent Gaussian 
with mean zero and variance $1$, $\mathcal{N}(0,1)$, then 
\beq
	N\left( F_N - \left( \log 2 + \beta^2 \right)\right)  \Rightarrow \mathcal{N}\big(-\frac12 \alpha, \alpha\big),
\eeq
where
\beq\label{eq:SKvari}
	\alpha=-\frac12  \log(1-4\beta^2) -2 \beta^2,
\eeq
and the convergence is in distribution as $N\to\infty$ (see also Section 11.4 in \cite{TalagrandBook2011v2}.)
It was also shown in \cite{AizenmanLebowitzRuelle} that for non-Gaussian disorder, the same limit theorem holds with some changes on the formula of $\alpha$. 
However, a limit theorem for the fluctuations in the low temperature regime still remains as an open question. 
The limit theorem is not known even for the zero temperature case. 

\bigskip

The SSK model was introduced by Kosterlitz, Thouless, and Jones \cite{KosterlitzThoulessJones} as a model that is easier to analyze than the SK model. 
Indeed in their paper, the authors evaluated the limit of the free energy explicitly though a rigorous proof was not supplied. 
The analogue of Parisi formula for the SSK model was obtained by Crisanti and Sommers \cite{CrisantiSommers}. 
The Parisi formula for the SSK model was later proved rigorously by Talagrand \cite{TalagrandParisiSpher} immediately after he proved the formula for the SK model. 
The Parisi formula can be evaluated explicitly for the Hamiltonian~\eqref{SSK Hamiltonian} (for the case with $2$-spin interactions without magnetic field) \cite{PanchekoTalagrand2007} and the resulting formula is same as one obtained in \cite{KosterlitzThoulessJones}: 
\beq \label{eq:Fbetadef}
	F_N \to F(\beta)=
		\begin{cases} 
			\beta^2 & \text{ if } 0< \beta < 1/2, \\
			2\beta - \frac{\log (2\beta) + 3/2}{2} & \text{ if } \beta > 1/2
		\end{cases}
\eeq
in expectation and also in distribution. 
For a general class of random variables $J_{ij}$, a corresponding result for the limiting free energy $F(\beta)$ was obtained in \cite{Guionnet-Maida}. 
The formula in \cite{Guionnet-Maida} was given in terms of R-transform and one can check that it is same as the one in Definition \ref{def:parameters} below.
Note that $F(\beta)$ in~\eqref{eq:Fbetadef} is $C^2$ but not $C^3$ at $\beta_c=\frac12$; the critical temperature is same as that of the SK model. 
The third-order transition also holds for SSK model with general random variables $J_{ij}$ as one can see in Definition \ref{def:parameters}. 

We remark that paper \cite{BenArousDemboGuionnet2001} studied the so-called ``soft'' spherical Sherrington--Kirkpatrick (SSSK) model and evaluated the almost sure limit of the free energy explicitly. 
The limit of the free energy for the SSSK model also shows a third order phase transition. 

In this paper, we obtain the limit theorem for the fluctuations of $F_N$ for the SSK model. 
We first state the result when the disorder random variables are Gaussian. Non-Gaussian case will be stated 
in Theorem~\ref{thm:spin2}. 

\begin{thm}\label{thm:spin1}
Let $J_{ij}$, $1\le i<j\le N$, be independent Gaussian random variables with mean $0$ and variance $1$, 
and set $J_{ji}=J_{ij}$. 
Let
\beq\label{eq:freeenergdef}
	F_N= F_N(\beta)= \frac1{N} \log \left[ \int_{S_{N-1}} e^{- \frac{\beta}{\sqrt{N}} \sum_{i \neq j} J_{ij} \sigma_i \sigma_j } \dd \omega_N(\bss) \right]
\eeq
be the free energy of the SSK model at inverse temperature $\beta$,  
where $\dd\omega_N$ is the normalized uniform measure on the sphere $S_{N-1} = \{ \bss \in \R^N : \| \bss \|^2 = N \}$.
Then the following holds as $N\to \infty$. Here $F(\beta)$ is defined in~\eqref{eq:Fbetadef}, and all the convergences are in distribution. 
\begin{enumerate}[(i)]
\item In the high temperature regime $0<\beta<\frac12$, 
\beq\label{eq:thmspin1high}
	N\left( F_N - F(\beta) \right) \Rightarrow \mathcal{N}\left(f, \alpha\right), 
\eeq
where 
\beq
	f=\frac14  \log(1-4\beta^2) -2 \beta^2, \qquad 
	\alpha=-\frac12  \log(1-4\beta^2) -2 \beta^2.
\eeq
\item In the low temperature regime $\beta>\frac12$, 
\beq\label{eq:thmspin1low}
	\frac1{\beta-\frac12} N^{2/3} \left( F_N - F(\beta) \right) \Rightarrow 
	TW_{1}, 
\eeq
where $TW_{1}$ is the GOE Tracy-Widom distribution. 
\end{enumerate}
\end{thm}

Hence the order of the fluctuations changes from $N^{-1}$ in the high temperature regime to $N^{-2/3}$ in the low temperature regime. 
In the high temperature regime, the fluctuations are asymptotically Gaussian with the same variance as in SK model 
(see~\eqref{eq:SKvari}). 
On the other hand, in the low temperature regime, the fluctuations are asymptotically same as 
those of the largest eigenvalue of a large random matrix from GOE (Gaussian orthogonal ensemble). 
The connection to the GOE Tracy-Widom distribution is apparent at zero temperature:
From~\eqref{eq:freeenergdef}, we may define the free energy at zero temperature (which is the formal limit 
of $\frac1{\beta}F_N(\beta)$ as $\beta\to \infty$) as 
\beq
	\widetilde{F}_N(\infty)=  \sup_{\|\bss\|^2=N } \frac1{N} \langle \bss, \frac{-J}{\sqrt{N}} \bss \rangle = \lambda_{\max} (N)
\eeq
where $\lambda_{\max}(N)$ is the largest eigenvalue of the random symmetric matrix $M=-J/\sqrt{N}$, 
$J=(J_{ij})_{i,j=1}^N$.
The random matrix $M$ is almost exactly an $N\times N$ GOE matrix except that the diagonal terms are zero. 
The fluctuations of the smallest and the largest eigenvalues are still given by the GOE Tracy-Widom distribution in the limit $N\to \infty$ \cite{So1999}: $N^{2/3}(\tilde{F}_N(\infty) - 2) \Rightarrow TW_{1}$. 
The theorem above shows that the same limit law for the fluctuations hold for all $\beta>\frac12$ after the change by the multiplicative factor $\beta-\frac12$. 
The proof of the theorem will show that for $\beta>\frac12$, the main contribution to $F_N$ comes from the largest eigenvalue of $-J$ and this holds not only for the leading asymptotic term of $F_N$ but also for the second term corresponding to the fluctuations. 
On the other hand, we will see in the proof that for $\beta<\frac12$, all of the eigenvalues of the random matrix $-J/\sqrt{N}$ contribute to $F_N$ in the form of the linear statistic $\sum_{i=1}^N g(\lambda_i)$ where $\lambda_i$ are the eigenvalues of $-J/\sqrt{N}$ for a specific function $g$.
It is a well-known result in random matrix theory that if the function $g$ is smooth in an open interval that contains the support of the limiting density function of the eigenvalues, then the linear statistic converges to the Gaussian distribution 
\cite{Johansson98, BY2005, BWZ2009, LP}.

For the SK model, there is no analytic result for the fluctuations in the low temperature regime. 
See \cite{ParisiRizzo2009} for some physical analysis and conjectures. 
Some numerical studies \cite{Andreanov, Boettcher} suggest that at zero temperature, the order of the fluctuations are smaller that $N^{-2/3}$ (\cite{Boettcher} suggests $N^{-3/4}$) and the limiting distribution is not the GOE Tracy-Widom distribution. 
See \cite{Chatterjee2009} (also \cite{ChatterjeeBook2014}) for a mathematical result on the upper bound of the order of the fluctuations. 

\medskip

It is interesting to consider the near critical case when $\beta$ depends on $N$ and satisfies $\beta\to \frac12$ as $N\to \infty$. 
Even though we do not show in this paper, it is possible to improve the proof in this paper to show that~\eqref{eq:thmspin1low} still holds when $\beta=\beta_N=\frac12 + N^{-\delta}$ with $\delta<\frac13$. 
It is expected that~\eqref{eq:thmspin1high} also holds when $\beta=\beta_N= \frac12-N^{-\delta}$ with $\delta<\frac13$. These will be discussed in a future paper. 
It is tempting to predict the critical window of $\beta$ in which the transition from the Gaussian distribution to the GOE Tracy-Widom distribution occurs by using Theorem~\ref{thm:spin1}. 
The theorem indicates that the leading-order term of the variance of $F_N$ is of order 
$-\frac1{N^2} \log(1-2\beta)$ for $\beta<1/2$ and of order $\frac1{N^{4/3}}(\beta-\frac12)^2$ for $\beta>1/2$, 
as $N\to \infty$ and $\beta\to \frac12$. 
By matching these orders, we are lead to speculate that the critical window of the temperature is $\beta= \frac12 + O(\frac{\sqrt{\log N}}{N^{1/3}})$.

Another model in which the Tracy-Widom distribution appears is the directed polymer in random environment (DPRE) 
in $1+1$ dimension. 
Recent impressive developments in the field show that for some specific choices of disorders, the fluctuations of the free energy 
are given by the GUE Tracy-Widom distribution for all $\beta>0$ \cite{AmirCorwinQuastel11, BCF, BCR, CorwinSeppalShen14, OConnellOrtmann14}. 
(Here GUE stands for Gaussian unitary ensemble.)
It was indeed shown previously in \cite{CaromaHu2002, CometsShigaYoshida2003} that the critical temperature for $1+1$ dimensional DPRE is $\beta_c=0$. 
Recently it was shown by Alberts, Khanin, and Quastel \cite{AlbertsKhaninQuastel14}
that the critical window is $\beta=O(N^{-1/4})$. More specifically if $\beta= BN^{-1/4}$, then 
the fluctuations of the free energy converge to a different distribution parametrized by $B$, called the crossover distribution that appears in the KPZ equation (see also \cite{MorenoSeppalValko14}). 
This regime is called the intermediate disorder regime in \cite{AlbertsKhaninQuastel14}. 

\bigskip

For non-Gaussian disorder random variables, we have the following universality result. 
Note that the disordered random variables are not necessarily identically distributed. 

\begin{thm}\label{thm:spin2}
Let $J_{ij}$, $1\le i\le j\le N$, be independent random variables satisfying the following conditions:
\begin{itemize}
\item All moments of $J_{ij}$ are finite and $\mathbb{E}[J_{ij}]=0$ for all $1\le i\le j\le N$.
\item For all $i<j$, $\mathbb{E}[J_{ij}^2]=1$, $\mathbb{E}[|J_{ij}|^3]=W_3$, and $\mathbb{E}[J_{ij}^4]=W_4$ for some constants $W_3, W_4\ge 0$. 
\item For all $i=1, \cdots, N$, $\mathbb{E}[J_{ii}^2]=w_2$ for a constant $w_2\ge 0$. 
\end{itemize}
Set $J_{ji}=J_{ij}$ for $i<j$. 
Define the free energy as~\eqref{eq:freeenergdef} with the sum replaced by $\sum_{i, j=1}^N J_{ij} \sigma_i \sigma_j$.
Then~\eqref{eq:thmspin1high} still holds after the changes 
\beq
	f=\frac14 \log(1-4\beta^2) + \beta^2(w_2-2)+ 2\beta^4(W_4-3),
	\quad 
	\alpha=-\frac12 \log(1-4\beta^2) +\beta^2(w_2-2)+ 2\beta^4(W_4-3).
\eeq
On the other hand,~\eqref{eq:thmspin1low} holds without any changes. 
\end{thm}

\bigskip

The starting point of the proofs of the above theorems is a simple integral formula of the partition function. 

\begin{lem}\label{inverse laplace}
Let $M$ be an $N\times N$ symmetric matrix with eigenvalues $\lambda_1\ge\cdots\ge\lambda_N$. 
Then  
\beq \label{integral representation0}
	\int_{S_{N-1}} e^{\beta \langle \bss, M \bss \rangle }\dd \omega_N(\bss) 
	= C_N  \int_{\gamma - \ii \infty}^{\gamma + \ii \infty} e^{\frac{N}{2} G(z)} \dd z, 
	\quad 
	G(z) = 2\beta z - \frac{1}{N} \sum_i \log (z - \lambda_i),
\eeq
where $\gamma$ is any constant satisfying $\gamma>\lambda_1$, the integration contour is the vertical line from $\gamma-\ii \infty$ to $\gamma+\ii \infty$, the $\log$ function is defined in the principal branch, and 
\beq
	C_N	= \frac{\Gamma(N/2)}{2\pi \ii (N\beta)^{N/2-1}}.
\eeq 
Here $\Gamma(z)$ denotes the Gamma function. 
\end{lem}

We may apply the method of steepest-descent to analyze the asymptotic behavior of the above integral as $N\to \infty$. 
The difficulty is that $G(z)$ is random since $M= -J/\sqrt{N}$ is a random matrix. 
Now random matrix theory tells us that the eigenvalues of random symmetric matrix $M$ have strong repulsions between them and as a consequence they are rigid in the sense that the eigenvalues are close to the deterministic locations determined by the quantiles of their limiting empirical distribution (i.e. semi-circle law). 
This rigidity of the eigenvalues allows us still to be able to apply the method of steepest-descent. 
The crucial technical ingredient here is precise estimate on the rigidity of the eigenvalues that was obtained recently by Erdos, Yin, and Yau \cite{EYY}. 
This precise rigidity estimate is one of the central achievements of the recent surge of advancements of our understanding of random matrices. 
After we obtained the above integral representation~\eqref{integral representation0}, analyzed them asymptotically, and obtained the results in this paper, we learned that the same integral representation was already obtained in the paper of Kosterlitz, Thouless, and Jones \cite{KosterlitzThoulessJones} in which they obtained the leading order term of the asymptotics by using the method of steepest-descent but without supplying rigorous estimates. 
Our analysis makes their work rigorous, and goes a step further and obtains the second asymptotic term giving the law of the fluctuations. 
The paper \cite{KosterlitzThoulessJones} also considered the case when the mean of the disorder is not necessarily zero. We plan to study this case in a separate paper by the same method as in this paper. 
A similar formula to the above Lemma also appeared in \cite{Mo} for the analysis of rank 1 real Wishart spiked model.

Since our analysis only relies on the above integral formula and the rigidity of the eigenvalues, the random matrix, $M$,  corresponding to the disorder random variables does not necessarily have to have independent entries (hence corresponding to the real Wigner matrices). 
We indeed obtain similar results for $M$ from orthogonal invariant ensembles or real sample covariance matrices, and for $M$ from complex Hermitian matrices, since the rigidity of the eigenvalues was proved for a wide variety of random matrices. 
In the next section, we state general results assuming some spectral properties of random matrices, and in the subsequent section, we list a few of random matrices, including the one corresponding to the SSK model, for which the results may apply.

\bigskip

The rest of paper is organized as follows. 
In Section~\ref{sec:main}, we introduce general conditions and state general results.
In Section \ref{sec:examples}, we illustrate the examples of random matrix ensembles that satisfy the general conditions. Theorems~\ref{thm:spin1} and~\ref{thm:spin2} follow from one of these examples. 
In Section~\ref{sec:integral}, we prove Lemma \ref{inverse laplace}. 
Sections \ref{sec:sub} and \ref{sec:sup} are the main technical part of this paper in which we analyze the integral representation in Lemma \ref{inverse laplace} asymptotically by using the method of steepest-descent. 
The high temperature regime is analyzed in Section \ref{sec:sub} and the low temperature regime is analyzed in Section \ref{sec:sup}. 
In Section~\ref{sec:third}, we prove Theorem \ref{thm:third}, on the third order phase transition of the free energy. Some technical details in Section \ref{sec:examples} are collected in the Appendix.

\medskip

\begin{rem}[Notational Remark 1]
Throughout the paper we use $C$ or $c$ in order to denote a constant that is independent of $N$. 
Even if the constant is different from one place to another, we may use the same notation $C$ or $c$ as long as it does not depend on $N$ for the convenience of the presentation.
\end{rem}

\begin{rem}[Notational Remark 2]
The notation $\Rightarrow$ denotes the convergence in distribution as $N\to \infty$. 
\end{rem}

\medskip 

We close this section by introducing the following terminology.  

\begin{defn}[High probability event]
We say that an $N$-dependent event $\Omega_N$ holds with high probability if, for any given $D > 0$, there exists $N_0 > 0$ such that
$$
\p (\Omega_N^c) \leq N^{-D}
$$
for any $N > N_0$.
\end{defn}

\subsubsection*{Acknowledgments}
We would like to thank Tuca Auffinger, Zhidong Bai, Paul Bourgade, Joe Conlon, Dmitry Panchenko, and Jian-feng Yao for several useful communications. 
Ji Oon Lee is grateful to the department of mathematics, University of Michigan, Ann Arbor, for their kind hospitality during the academic year 2014--2015.
The work of Jinho Baik was supported in part by NSF grants DMS1361782. The work of Ji Oon Lee was supported in part by Samsung Science and Technology Foundation project number SSTF-BA1402-04.

\section{General results} \label{sec:main}

\subsection{Definitions and conditions}\label{sec:conditions}

\begin{defn}\label{def:partition}
For an $N\times N$ real random symmetric matrix $M=(M_{ij})_{i,j=1}^N$, 
we define the partition function at inverse temperature $\beta>0$ by 
\beq \label{partition function'}
	Z_N =Z_N(\beta)= \int_{S_{N-1}} e^{\beta \langle \bss, M\bss \rangle } 
	\dd \omega_N(\bss),
\qquad \langle \bss, M\bss \rangle = \sum_{i,j=1}^N M_{ij} \sigma_i\sigma_j, 
\eeq
where $\dd\omega_N$ is the normalized uniform measure on the sphere $S_{N-1} = \{ \bss \in \R^N : \| \bss \|^2 = N \}$.
The free energy $F_N$ is defined by
\beq \label{eq:freeenergydef}
	F_N = F_N(\beta)= \frac{1}{N} \log Z_N.
\eeq 
\end{defn}

The free energy~\eqref{SSK free e} for the SSK model corresponds to the case when $M=-J/\sqrt{N}$ where $J$ is a symmetric random matrix whose diagonal entries are zero and the entries below the diagonal are independent and identical random variables of mean $0$ and variance $1$. 
As mentioned in the previous section, we prove the limit theorem for the fluctuations for more general random symmetric matrices. 
Precise conditions on $M$ in terms of its eigenvalues will be stated shortly below and these conditions are shown to be satisfied for Wigner matrices, invariant ensembles, and sample covariance matrices in Section~\ref{sec:examples}.  

We also consider Hermitian matrices. 
\begin{defn}
For an $N\times N$ complex random Hermitian matrix $M$, we define 
\beq \label{partition function' complex}
	Z_N =Z_N(\beta)= \int_{\C S_{N-1}} e^{\beta \langle \bss, M\bss \rangle } 
	\dd \omega_N(\bss),
\qquad \langle \bss, M\bss \rangle = \sum_{i,j=1}^N M_{ij} \overline{\sigma_i}\sigma_j, 
\eeq
where $\C S_{N-1}= \{ \bss \in  \C^N : \| \bss \|^2 = N \} \simeq S_{2N-1}$ and $\omega_N(\bss)$ is the uniform measure on $\C S_{N-1}$. 
The free energy $F_N$ is defined by the same formula~\eqref{eq:freeenergydef}. 
\end{defn}

For a symmetric or Hermitian matrix $M$, let $\lambda_1 \geq \lambda_2 \geq \cdots \geq \lambda_N$ denote the eigenvalues of $M$. 
We now list four conditions for the eigenvalues of random matrix $M$ under which the general theorems are proved. 

Let $\nu_N:= \frac{1}{N} \sum_{j=1}^N \delta_{\lambda_j}$ 
denote the empirical spectral measure of $M$. 
We assume that $\nu_N$ converges weakly to a probability measure $\nu$. 
Our first condition is the regularity of the limiting spectral measure $\nu$ in the following sense.

\begin{cond}[Regularity of measure] \label{cond:regular}
Suppose that the empirical spectral measure $\nu_N$ of $M$ converges weakly to a probability measure $\nu$ that satisfies the following properties:

\begin{itemize}
\item $\nu$ is supported on an interval $[C_-, C_+]$ and is positive on $(C_-, C_+)$.

\item $\nu$ is absolutely continuous and $\frac{\dd \nu}{\dd x}$ exhibits square root decay at the upper edge, i.e.,
\beq\label{eq:sqrtb}
 \frac{\dd \nu}{\dd x}(x) = s_{\nu} \sqrt{C_+-x} \left( 1 + O(C_+ -x) \right) \quad \text{ as } x \nearrow C_+
\eeq
for some $s_{\nu} > 0$.
\end{itemize}
\end{cond}

The second condition concerns the rigidity of the eigenvalues. This is the key assumption. 

\begin{cond}[Rigidity of eigenvalues] \label{cond:rigidity}
For a positive integer $k \in [1, N]$, let $\hat k := \min \{ k, N+1-k \}$. Let $\gamma_k$ be the classical location defined by
\beq\label{eq:classicallocationdef}
\int_{\gamma_k}^{\infty} \dd \nu = \frac{1}{N} \left( k - \frac{1}{2} \right).
\eeq
Assume that for any $\epsilon > 0$
\beq \label{rigidity}
| \lambda_k - \gamma_k | \leq \hat k^{-1/3} N^{-2/3 + \epsilon}
\eeq
holds for all $k$ with high probability.
\end{cond}

We remark that under the assumption~\eqref{eq:sqrtb}, the classical location $\gamma_k$ satisfies the estimate
\beq\label{eq:classicalfromC}
C^{-1} k^{2/3} N^{-2/3} \leq |C_+ - \gamma_k| \leq C k^{2/3} N^{-2/3}
\eeq
for some constant $C > 1$ that is independent of $N$.

The third condition is about the linear statistics of the eigenvalues and it is used in the analysis 
of the high temperature case $\beta < \beta_c$.

\begin{cond}[Linear statistics of the eigenvalues] \label{cond:linear}
Assume that for every function $\varphi:\R\to \R$ that is analytic in an open neighborhood of $[C_-, C_+]$ and has compact support, the random variable
\beq
\caN_{\varphi} := \sum_i \varphi(\lambda_i) - N \int_{C_-}^{C_+} \varphi(\lambda) \dd \nu(\lambda)
\eeq
converges in distribution to a Gaussian random variable. The mean and the variance of this Gaussian random variable are denoted by 
$\mean(\varphi)$ and $V(\varphi)$, respectively, and they depend only on $\varphi$ restricted on the support of $\nu$.
\end{cond}

The fourth condition is the convergence to the Tracy-Widom distribution of the largest eigenvalue. 
This will be used in the analysis of the supercritical case $\beta > \beta_c$. 

\begin{cond}[Tracy-Widom limit of the largest eigenvalue] \label{cond:edge}
Let $s_{\nu}$ be the constant appearing in~\eqref{eq:sqrtb}. 
Assume that the rescaled largest eigenvalue $(s_{\nu} \pi)^{-2/3} N^{2/3} (\lambda_1 - C_+)$ converges in distribution to the (GOE or GUE) Tracy-Widom distribution (depending on whether the matrices are symmetric or Hermitian). 
\end{cond}

We denote by $TW_1$ and $TW_2$ GOE and GUE Tracy-Widom random variables 
and by $F_1$ and $F_2$ their cumulative distribution functions, respectively. 

\begin{rem}
For invariant ensembles, the support of the measure $\nu$ in Condition~\ref{cond:regular} may consist of several disjoint intervals, 
In this case, Condition~\ref{cond:linear} does not hold in general; the variance of the linear statistics is a quasi-periodic function in $N$ and does not converge. 
See \cite{Pa2006}.
On the other hand, Condition~\ref{cond:edge} is known to hold for multi-inverval cases as well. 
In this paper we choose to simplify the situation by assuming that the support consists of a single interval; one can still prove some parts of the theorems in the next subsection without assuming this condition. 
\end{rem}

\subsection{Results for symmetric matrices}\label{sec:results}

We first define the critical inverse temperature. 

\begin{defn}[Critical $\beta$] \label{def:beta_c}
Assume that Condition \ref{cond:regular} holds.
Define 
\beq\label{eq:defbetac}
	\beta_c^{sym} = \beta_c = \frac{1}{2} \int_{C_-}^{C_+} \frac{\dd \nu(x)}{C_+ -x}.
\eeq
\end{defn}

Note that $\beta_c$ is finite due to the square root decay of $\frac{\dd \nu}{\dd x}$ at the upper edge. 
The exact values of $\beta_c$ can be evaluated for some random matrix ensembles (see Section~\ref{sec:examples}) as follows. 
\begin{itemize}
\item $\beta_c = \frac12$ for Wigner matrices (with the choice $C_+=2$): SSK model belongs to this example.
\item $\beta_c = \frac14 Q'(C_+)$ for invariant ensembles with potential $Q$.
\item $\beta_c = \frac1{2 \sqrt{C_+}}$ for sample covariance matrices.
\end{itemize}

The first main result is the following.
\begin{thm}[Third order phase transition] \label{thm:third}
Consider an ensemble of symmetric matrices satisfying Conditions~\ref{cond:regular},~\ref{cond:rigidity},~\ref{cond:linear}, and~\ref{cond:edge}.
There is a function $F:(0,\infty)\to [0, \infty)$ such that the free energy $F_N$ satisfies
\beq
	F_N(\beta) \Rightarrow F(\beta)
\eeq
in distribution as $N\to \infty$, for $\beta\neq \beta_c$.
The function $F(\beta)$ is $C^2$ but its third derivative $\partial^3_{\beta} F(\beta)$ is discontinuous at $\beta=\beta_c$.
This function is defined explicitly in Definition~\ref{def:parameters} below. 
\end{thm}

The next two results are about the fluctuations of $F_N$. 

\begin{thm}[high temperature case] \label{thm:sub}
Consider an ensemble of symmetric random matrices satisfying Conditions~\ref{cond:regular},~\ref{cond:rigidity}, and~\ref{cond:linear}. 
Then for $\beta<\beta_c$, 
\beq
	N \left( F_N (\beta)- F(\beta) \right)  \Rightarrow \mathcal{N}(\ell, \sigma^2)
\eeq
where the constants $\ell \equiv \ell(\beta)$ and $\sigma^2 \equiv \sigma^2(\beta)$ are defined in Definition~\ref{def:parameters} below.
\end{thm}

\begin{thm}[low temperature case] \label{thm:sup}
Consider an ensemble of symmetric matrices satisfying Conditions~\ref{cond:regular},~\ref{cond:rigidity}, 
and~\ref{cond:edge}. 
Then for $\beta>\beta_c$, 
\beq \label{sup relation}
	\frac1{(s_{\nu} \pi)^{2/3} \left( \beta - \beta_c \right)} N^{2/3} \left( F_N(\beta) - F(\beta) \right) \Rightarrow TW_1. \eeq
Here the constant $s_{\nu}$ is the one in Condition~\ref{cond:regular}.
\end{thm}

\begin{rem} \label{rem:fraction}
For a given random symmetric matrix $M$, $\frac1{Z_N} e^{\beta H(\bss)}$ defines a probability measure on the sphere $S_{N-1}$. 
It is interesting to study how far a random point $\bss$ on $S_{N-1}$ under this probability measure is from the eigenspace for the largest eigenvalue of the matrix. 
One such a measurement is the random variable 
\beq \label{condensate fraction00}
\mathbb{E}_M \left[ |\langle\bss,\bsv_1\rangle|^2 \right] = \frac{1}{N} \int_{S_{N-1}} |\langle \bss, \bsv_1 \rangle|^2 \frac{e^{\beta H(\bss)}}{Z_N} \dd \omega_N (\bss)
\eeq
where $\bsv_1$ is an $\ell_2$-normalized eigenvector associated with $\lambda_1$. 
Note that due to the absolute value in $|\langle\bss,\bsv_1\rangle|$,~\eqref{condensate fraction00} 
does not depend on the choice of the eigenvector $\bsv_1$. 
It is easy to check that $\mathbb{E}_M \left[ |\langle\bss,\bsv_1\rangle|^2 \right] = -\frac{1}{\beta N Z_N} \frac{\partial Z_N}{\partial \lambda_1}$ and, from this formula, 
it is straightforward to prove that  $\mathbb{E}_M \left[ |\langle\bss,\bsv_1\rangle|^2 \right]=O(N^{-1})$ for $\beta < \beta_c$ by modifying the analysis in Section~\ref{sec:sub} for the proof of Theorem~\ref{thm:sub}. 
This is consistent with the fact that for the high temperature case, all eigenvalues contribute to the free energy and the fluctuations of $F_N$ come from the fluctuations of certain linear statistics of all of the eigenvalues. 
On the other hand, for the super-critical case when $\beta>\beta_c$, we expect that $\mathbb{E}_M \left[ |\langle\bss,\bsv_1\rangle|^2 \right]$ converges to a constant that depends on $\beta$.
\end{rem}

\bigskip

The constants appearing in the above theorems 
are given as follows. 
Note that $h(s):=\frac12 \int_{C_-}^{C_+} \frac{\dd \nu(x)}{s - x}$ is a decreasing function for real $s>C_+$ and $h(s)\to 0$ as $s\to +\infty$. 
Moreover, by the definition~\ref{def:beta_c} of $\beta_c$, $h(s)\to \beta_c$ as $s\searrow C_+$. 
Hence for $\beta<\beta_c$, there is a unique $\widehat\gamma\equiv \widehat\gamma(\beta) \in (C_+, \infty)$ that satisfies
\beq \label{stieltjes gamma}
	\frac12 \int_{C_-}^{C_+} \frac{\dd \nu(x)}{\widehat\gamma - x} = \beta.
\eeq
Note that $\widehat\gamma(\beta)$ is an decreasing function in $\beta\in (0, \beta_c)$ and $\widehat\gamma(\beta) \searrow C_+$ as $\beta\nearrow \beta_c$. 

\begin{defn}\label{def:parameters}
Define 
\beq \label{L sub010}
F(\beta) = \beta \widehat\gamma(\beta) - \frac{1}{2} \left( \int_{C_-}^{C_+} \log (\widehat\gamma(\beta) - k) \dd \nu(k) + 1 + \log (2\beta) \right), 
\qquad \beta<\beta_c,
\eeq
and 
\beq \label{eq:L sup}
F(\beta) = \beta C_+ - \frac{1}{2} \left( \int_{C_-}^{C_+} \log (C_+ - k) \dd \nu(k) + 1 + \log (2\beta) \right), 
\qquad \beta\ge \beta_c.
\eeq
Furthermore, for $\beta<\beta_c$, define
\beq 
	\ell(\beta) = \ell_1(\beta) -\frac1{2}\mean(\varphi), \qquad \sigma^2(\beta) = \frac{1}{4} V(\varphi), 
\eeq 
where 
\beq \label{eq: ell one def}
	\ell_1(\beta) = \log (2\beta) - \frac{1}{2} \log \left( \int_{C_-}^{C_+} \frac{\dd \nu(k)}{(\widehat\gamma(\beta)-k)^2} \right) 
\eeq
and
$\mean(\varphi)$ and $V(\varphi)$ are defined in Condition~\ref{cond:linear} with 
\beq
	\varphi(x)=\log(\widehat{\gamma}(\beta)-x).
\eeq
\end{defn}

\subsection{Results for Hermitian matrices} \label{sec:hermitian}

All of the previous results hold for Hermitian matrices after the following simple changes:
\begin{enumerate}
\item The critical value in Definition~\ref{def:beta_c} is changed to $\beta_c^{H} = 2 \beta_c^{sym}$. 
\item Theorems~\ref{thm:third} and~\ref{thm:sub} 
hold without any changes.
\item Theorem~\ref{thm:sup} holds with $TW_1$ replaced by $TW_2$ in~\eqref{sup relation}. 
\item In Definition~\ref{def:parameters}, all terms remain the same after the change that $\beta$ is replaced by $\beta/2$. 
For example, $F^H(\beta)=F(\beta/2)$.
\end{enumerate}

\section{Examples} \label{sec:examples}

We list some random matrix ensembles, which satisfy the conditions in Subsection \ref{sec:conditions}, and hence to which the general results in Subsections \ref{sec:results} and \ref{sec:hermitian} apply.

\subsection{Wigner matrix} \label{sec:Wigner}

A real Wigner matrix is an $N \times N$ real symmetric matrix $M$ whose upper triangle entries $M_{ij}$ $(i \leq j)$ are independent real random variables satisfying the following conditions:
\begin{itemize}
\item The entries are centered, i.e., $\E [ M_{ij}] = 0$ for all $i,j$.
\item Their variances satisfy that $\E [ |M_{ij}|^2] = \frac1{N}$ for $i \neq j$ 
and $\E [|M_{ii}|^2] = \frac{w_2}{N}$ for a constant $w_2 \geq 0$. 
\item For any integer $p > 2$, $\E [ |M_{ij}|^p ] = O(N^{-p/2})$. Moreover, $\E [|M_{ij}|^3]$ and $\E [|M_{ij}|^4]$ do not depend on $i, j$.
\end{itemize}

A complex Wigner matrix is an $N \times N$ complex Hermtian matrix $M$ whose real and imaginary parts of the entries are all independent, modulo the Hermitian condition, and satisfy the same moments conditions as above and an extra condition that $\E [(M_{ij})^2] = 0$ for $i \neq j$.

\begin{rem}
We note that some of the Conditions in Subsection \ref{sec:conditions} are still satisfied even if some of the conditions on the definition of Wigner matrices, such as the existence of all moments, are relaxed. However, we content with the above definition of Wigner matrices so that all of the four Conditions in the previous section are simultaneously satisfied.
Similar remark also applies to the random matrix ensembles in the next two subsections. 
\end{rem}

For real and complex Wigner matrices, the following are known:

\begin{enumerate}
\item Condition~\ref{cond:regular} (Regularity)

For both real and complex case, the limiting spectral measure is given by the semicircle law, 
\beq
\frac{\dd \nu}{\dd x} (x) = \frac{1}{2\pi} \sqrt{4-x^2},\qquad -2\le x\le 2. 
\eeq 
Hence Condition~\ref{cond:regular} is satisfied with $C_+=2$ and $s_\nu=\frac1{\pi}$. 

\item Condition~\ref{cond:rigidity} (Rigidity)

Condition~\ref{cond:rigidity} was proved in \cite{EYY} for both real and complex cases. 

\item Condition~\ref{cond:linear} (Linear statistics)

Condition~\ref{cond:linear} was proved in \cite{BY2005}. See also \cite{BWZ2009} and \cite{LP} for non-analytic test functions. 
The mean $\mean(\varphi)$ and the variance $V(\varphi)$ for function $\varphi$ are as follows. Let
\beq
	w_2 :=N\E[|M_{11}|^2], \qquad W_4 := N^2 \E [|M_{12}|^4].
\eeq
We have $w_2=2, W_4=3$ for GOE, and $w_2=1, W_4=2$ for Gaussian unitary ensemble (GUE). 
Set
\beq \label{Chebyshev formula}
	\tau_\ell(\varphi)= \frac1{\pi} \int_{-2}^2 \varphi(x) \frac{T_\ell(x/2)}{\sqrt{4-x^2}} \dd x
	= \frac1{2\pi} \int_{-\pi}^\pi \varphi(2\cos\theta) \cos(\ell\theta) \dd \theta
\eeq
for $\ell=0,1,2,\cdots$, 
where $T_\ell(t)$ are the Chebyshev polynomials of the first kind; 
$T_0(t)=1$, $T_1(t)=t$, $T_2(t)=2t^2-1$, $T_3(t)=4t^3-3t$, $T_4(t)=8t^4-8t^2+1$, etc.

The mean and the variance for the real case are 
\beq \label{Wigner mean}
	\mean(\varphi) = \mean_{GOE}(\varphi)  + \tau_2(\varphi)  (w_2 -2) + \tau_4(\varphi) (W_4 -3)
\eeq
and 
\beq \label{Wigner variance}
	V(\varphi) = V_{GOE}(\varphi)  + \tau_1(\varphi)^2 (w_2-2) + 2 \tau_2(\varphi)^2 (W_4-3), 
\eeq
respectively, where
\beq \label{GOE mean}
	\mean_{GOE}(\varphi) = \frac14\left[ \varphi(2)+\varphi(-2) \right] - \frac12 \tau_0(\varphi)
\eeq
and
\beq \label{GOE variance}
\begin{split}
	V_{GOE}(\varphi) &= 2\sum_{\ell=1}^\infty \ell \tau_\ell(\varphi)^2 \\
	&= \frac{1}{2\pi^2} \int_{-2}^2 \int_{-2}^2 \left( \frac{\varphi(\lambda_1) - \varphi(\lambda_2)}{\lambda_1 - \lambda_2} \right)^2 \frac{4-\lambda_1 \lambda_2}{\sqrt{4-\lambda_1^2} \sqrt{4-\lambda_2^2}} \dd \lambda_1 \dd \lambda_2.
\end{split}
\eeq

For the complex case, 
\beq \label{Wigner mean complex}
	\mean^H(\varphi) = \mean_{GUE}(\varphi)  + \tau_2(\varphi) (w_2 -1)  + \tau_4(\varphi) (W_4 -2)
\eeq
and 
\beq \label{Wigner variance complex}
	V^H(\varphi)= V_{GUE}(\varphi)  + \tau_1(\varphi)^2  (w_2-1)+ 2\tau_2(\varphi)^2 (W_4-2) , 
\eeq
respectively, where
\beq \label{GUE variance}
	\mean_{GUE}(\varphi) = 0, \qquad V_{GUE}(\varphi) = V_{GOE}(\varphi)/2.
\eeq

\item Condition~\ref{cond:edge} (Tracy-Widom limit)

The Tracy-Widom distribution limit of the largest eigenvalue was proved in \cite{So1999, TV2010, EYY} for both real and complex cases. 

\end{enumerate}

\bigskip

We can evaluate the various constants appearing in the theorems in Subsection~\ref{sec:results} explicitly
and obtain following for real Wigner matrices. 
\begin{enumerate}[(i)]
\item $\beta_c=\frac12$.
\item The limit of the free energy is
\beq
	F(\beta)=
	\begin{cases}
	\beta^2 & \text{ if } 0<\beta < 1/2 \\
	2\beta - \frac{\log (2\beta) + 3/2}{2} & \text{ if } \beta > 1/2.
	\end{cases}
\eeq
\item For $\beta<\frac12$, $N \left(F_N(\beta) - F(\beta)\right) \Rightarrow \mathcal{N}(\ell, \sigma^2)$ where
\beq \label{eq:ell Wigner}
	\ell = \frac{1}{4}  \left( \log (1-4\beta^2)+ 4\beta^2 (w_2-2) + 8\beta^4 (W_4-3) \right)
\eeq
and
\beq \label{eq:sigma Wigner}
	\sigma^2 =  \frac{1}{2} \left(  -\log (1-4\beta^2)+ 2\beta^2 (w_2 -2) + 4 \beta^4 (W_4 -3)  \right) .
\eeq

\item For $\beta>\frac12$, 
$(\beta-\frac12)^{-1}N^{2/3} \left( F_N(\beta)-F(\beta)\right) \Rightarrow TW_1$.
\end{enumerate}
See Appendix \ref{appendix Wigner} for the detail.
This proves Theorem~\ref{thm:spin1} and Theorem~\ref{thm:spin2}. 

\bigskip

For complex Wigner matrices, we have the following changes:
(i) $\beta_c^H = 1$ and (ii) $L^H(\beta)=L(\beta/2)$. For (iii), 
\beq
	\ell = \frac{1}{2} \left( \log (1-\beta^2)+ \beta^2 (w_2-1) +  \frac{\beta^4}2  (W_4-2) \right) 
\eeq
and 
\beq
	\sigma^2 = -\log (1-\beta^2)+ \beta^2 (w_2 -1) + \frac{\beta^4}2 (W_4 -2).
\eeq
for $\beta<1$. For (iv), $\beta$ is replaced by $\beta/2$ and $TW_1$ by $TW_2$.

\subsection{Invariant ensemble} \label{sec:invariant}

The orthogonal invariant ensemble associated with potential $Q:\R\to\R$ is defined by the density
\beq
	\caP (M) = \frac{1}{Z} \exp (- \frac{N}2 \Tr Q(M)) 
\eeq
on the space of $N\times N$ real symmetric matrices and $Z$ is the normalization constant.
Similarly, the unitary invariant ensemble associated with potential $Q:\R\to\R$ is defined by the density
\beq
	\caP^H (M) = \frac{1}{Z^H} \exp (- N \Tr Q(M))
\eeq
on the space of complex Hermitian matrices, where $Z^H$ is the normalization constant. 
The GOE and the GUE correspond to the choice $Q(x) = x^2 /2$.
We assume that $Q$ is a polynomial of even degree with positive leading coefficient.
Many of the results below hold true for more general $Q$ but we restrict to this class of $Q$ for the convenience of the presentation. 
Furthermore, we assume that the associated equilibrium measure $\nu$ is of form 
\beq\label{eq: nu assump}
	\dd \nu(x) = \lone_{[C_-, C_+]}(x) h(x) \sqrt{(C_+-x)(x-C_-)}\dd x
\eeq
for constants $C_+>C_-$ and for a function $h(x)$ that is real analytic and positive in an open set containing the interval $[C_-, C_+]$. 
Recall that for general $Q$, the support of $\nu$ may consist of several intervals. 
Here we make the single interval assumption in order to use the central limit theorem for linear statistics. 
On the other hand, the square root behavior at the end points of the support holds for generic $Q$ 
\cite{Kuijlaars-McLaughlin00}. 
We remark that~\eqref{eq: nu assump} is guaranteed if $Q$ is convex. 

The following are known: 

\begin{enumerate}
\item Condition~\ref{cond:regular} (Regularity)

This holds from the assumption~\eqref{eq: nu assump} since the limiting spectral measure is given by the equilibrium measure. 
We remark that, from the variational condition on $\nu$, the following relation holds: 
\beq \label{eq:half Q pr}
	\frac 12 Q'(C_+) = \int_{C_-}^{C_+} \frac{\dd \nu(x)}{C_+ - x}.
\eeq

\item Condition~\ref{cond:rigidity} (Rigidity)

Condition~\ref{cond:rigidity} is proved in \cite{BEY} (see also \cite{BEY2012, BEY2014}).

\item Condition~\ref{cond:linear} (Linear statistics)

The linear statistics of eigenvalues was proved in \cite{Johansson98}.
The mean $\mean(\varphi)$ and the variance $V(\varphi)$ for function $\varphi$ are as follows.
For orthogonal invariant ensemble, we have 
\beq
	\mean_{OE}(\varphi)= \frac14 \left( \varphi(C_-)+\varphi(C_+) \right) - \int_{C_-}^{C_+} \frac{\varphi(\lambda) U_Q(\lambda)}{\sqrt{(\lambda-C_-)(C_+-\lambda)}} d\lambda
\eeq
for some rational function $U_Q$ which depends on $Q$.
The formula of $U_Q$ is complicated and is given in (3.54) of \cite{Johansson98}. 
It is, in particular, given by $U_Q(x)= \frac1{2\pi}$ when $Q(x)= \frac12 x^2$. 
On the other hand, 
\beq
	V_{OE}(\varphi) = V_{GOE}(\Phi), \qquad \Phi(x)= \varphi\left( \frac{C_-+C_+}2 + \frac{C_+-C_-}4 x \right) ,
\eeq
where $V_{GOE}$ is the variance~\eqref{GOE variance} for the GOE case. 
The change from $\varphi$ to $\Phi$ comes from the simple translation of the interval $(C_-, C_+)$ to $(-2,2)$. 
Observe that the variance $V_{OE}(\varphi)$ does not depend on the structure of the equilibrium measure except for the end points $C_-$ and $C_+$.On the contrary, the mean $\mean_{OE}(\varphi)$ depends on the full structure of the equilibrium measure. 

For unitary invariant ensemble, we have 
\beq
	\mean_{UE}(\varphi)= 0, \qquad V_{UE}(\varphi)= V_{OE}(\varphi)/2. 
\eeq

We remark that Condition~\ref{cond:linear} is not always valid if $\supp \nu$ consists of multiple disjoint intervals. See \cite{Pa2006} for more detail.

\item Condition~\ref{cond:edge} (Tracy-Widom limit)

Condition~\ref{cond:edge} was proved in \cite{Pastur-Shcherbina97, Bleher-Its99, Deift-Kriecherbauer-McLaughlin-Venakides-Zhou99, Deift-Gioev07a, BEY}.

\end{enumerate}

\bigskip
 
We can easily check from~\eqref{eq:half Q pr} that the critical value is 
\beq
	\beta_c^{OE} = \frac{1}{2} \int_{C_-}^{C_+} \frac{\dd \nu(x)}{C_+ -x} = \frac{1}{4} Q'(C_+), 
\qquad \beta_c^{UE}= \frac12 Q'(C_+)
\eeq
for orthogonal ensembles and unitary ensembles, respectively. 
Other constants appearing in the main theorems of this paper in the previous section may be evaluated for a given potential once the equilibrium measure is obtained.

\subsection{Sample covariance matrix} \label{sec:sample}

Let $X$ be a $K \times N$ matrix whose entries are independent real random variables satisfying the following conditions:
\begin{itemize}
\item The entries are centered, i.e., $\E [X_{ij}] = 0$. 
\item Their variances satisfy that $\E [ |X_{ij}|^2] = \frac{1}{N}$. 
\item For any integer $p > 2$, $\E [|X_{ij}|^p] = O(N^{-p/2})$. Moreover, $\E [ |X_{ij}|^3]$ and $\E [|X_{ij}|^4]$ do not depend on $i, j$.
\end{itemize} 
A sample covariance matrix $M$ is a random matrix of the form $M = X^* X$. 
When $X$ is a complex matrix, then we assume, in addition, that $\E [(X_{ij})^2] = 0$.
We assume further that $K \equiv K(N)$ with
\begin{align}
\frac{K}{N} \to d \in [1, \infty)
\end{align}
as $N \to \infty$.

The following are known.

\begin{enumerate}
\item Condition~\ref{cond:regular} (Regularity)

The limiting spectral measure is the Marchenko-Pastur distribution given by
\beq\label{eq:MP law}
	\dd \nu(x) = \frac{1}{2\pi} \frac{\sqrt{(C_+ -x)(x- C_-)}}{x} \dd x, 
\eeq
with support $[C_-, C_+]=[(\sqrt{d}-1)^2, (\sqrt{d}+1)^2]$.
Hence $s_{\nu} = \frac{d^{1/4}}{\pi (\sqrt{d}+1)^2}$.

\item Condition~\ref{cond:rigidity} (Rigidity)

Condition \ref{cond:rigidity} was proved in \cite{PY}. 

\item Condition~\ref{cond:linear} (Linear statistics)

Condition \ref{cond:linear} was proved in \cite{BS2004}. 
For non-analytic test functions, see \cite{BWZ2010} and \cite{LP}. 
The mean $\mean(\varphi)$ and the variance $V(\varphi)$ for function $\varphi$ are as follows: 
see (1.3)--(1.5) in \cite{BWZ2010}, (5.13) of \cite{BS2004}, and (4.28) of \cite{LP}.
Let $W_4 = N^2 \E[ |X_{11}|^4]$.
Set 
\beq \label{eq:sc linear stat chang}
	\Phi(x)= \varphi\left( \frac{C_-+C_+}2 + \frac{C_+-C_-}4 x \right)
	= \varphi(d+1+x\sqrt{d}). 
\eeq
For real sample covariance matrix,
\beq \label{sample variance}
	\mean(\varphi)=\mean_{GOE}(\Phi) - (W_4-3) \tau_2(\Phi), 
	\qquad 
	V(\varphi) = V_{GOE}(\Phi) + (W_4 -3)\tau_1(\Phi)^2. 
\eeq
For complex sample covariance matrix, 
\beq
	\mean^{comp}(\varphi)= - (W_4-2) \tau_2(\Phi), \qquad V^{comp}(\varphi) =V_{GUE}(\Phi) +  (W_4-2) \tau_1(\Phi)^2. 
\eeq

\item Condition~\ref{cond:edge} (Tracy-Widom limit)

Condition~\ref{cond:edge} was proved in \cite{Soshnikov2002, Pe2009, Wang2012, PY}. 

\end{enumerate}

Various constants can be evaluated explicitly and we obtain the following. 
See Appendix \ref{appendix sample} for the detail.
\begin{enumerate}[(i)]
\item $\beta_c=\frac1{2 \sqrt{C_+}} = \frac{1}{2(\sqrt d +1)}$.
\item The limit of the free energy per particle is
\beq \label{L sample covariance}
	F(\beta)=
	\begin{cases}
	- \frac{d}{2\beta} \log (1-2\beta), & \beta < \frac{1}{2(\sqrt d +1)}, \\
	(1+\sqrt d)^2 - \frac{1}{2\beta} \left( (1+\sqrt d) + d \log \frac{\sqrt d}{1+\sqrt d} - \log \frac{1}{1+\sqrt d} + \log (2\beta) \right), &\beta >\frac{1}{2(\sqrt d +1)}.
	\end{cases}
\eeq

\item For $\beta<\beta_c$, 
$N(F_N(\beta) - F(\beta)) \Rightarrow \mathcal{N}(\ell, \sigma^2)$ where 
where $L$ is given in \eqref{L sample covariance} and
\beq
	\ell= \frac1{4\beta} \left( \log(1-4B^2) - 4B^2(W_4-3) \right), \quad \sigma^2= \frac1{2\beta^2} \left( -\log(1-4B^2)+B^2(W_4-3) \right) 
\eeq
where we set 
\beq
	B= \frac{\beta\sqrt{d}}{1-2\beta}. 
\eeq
Note that $B<1/2$ for $\beta<\beta_c$ and $B>1/2$ for $\beta>\beta_c$.

\item For $\beta>\beta_c$,
$\frac{(1+\sqrt d)^{4/3}}{d^{1/6}} (\beta-\beta_c)^{-1} N^{2/3} \left( F_N(\beta)-F(\beta)  \right) \Rightarrow TW_1$.
\end{enumerate}

\bigskip

For complex sample covariance matrices, $\beta_c^H= \frac1{\sqrt{d}+1}$, $L^H(\beta)=L(\beta/2)$, and 
\beq
	\ell= \frac1{2\beta} \left( \log(1-4B^2) - 4B^2(W_4-2) \right), 
\quad 
	\sigma^2= \frac1{\beta^2} \left( -\log(1-4B^2) + 4B^2(W_4-2) \right) .
\eeq

\section{Integral representation of the partition Function} \label{sec:integral}

Our starting point in the analysis is the integral representation of the free energy given in Lemma \ref{inverse laplace}. We prove it here. 
As mentioned in Introduction, this formula was also obtained in \cite{KosterlitzThoulessJones}, and 
a similar formula appeared in \cite{Mo}.

\begin{proof}[Proof of Lemma \ref{inverse laplace}]
Let $S^{N-1}=\{ \bsx\in \R^N : \|\bsx\|=1\}$, the unit sphere in $\R^N$, and let $\dd \Omega$ be the surface area  measure on $S^{N-1}$. 
Hence $\frac{\dd \Omega}{|S^{N-1}|}$ is the uniform measure on $S^{N-1}$. 
We denote the left-hand side of~\eqref{integral representation0} by $Z_N$ as in~\eqref{partition function'}.
By change of variables, 
\beq \label{partition function}
	Z_N = \frac{1}{|S^{N-1}|} \int_{S^{N-1}} e^{\beta N \langle \bsx, M\bsx \rangle} \dd \Omega.
\eeq
We diagonalize $M$ and let $M = O^T D O$
for an orthogonal matrix $O$ and a diagonal matrix $D = \diag (\lambda_1, \lambda_2, \cdots, \lambda_N)$. 
Since $\langle \bsx , M\bsx\rangle = \langle O\bsx, D O\bsx\rangle$ and $O$ is orthogonal, we find 
after the changes of variables $\bsx\mapsto O^{-1} \bsx$ that
\beq
	Z_N
	= \frac{1}{|S^{N-1}|} \int_{S^{N-1}} e^{\beta N \langle \bsx, D \bsx \rangle} \dd \Omega 
	= \frac{1}{|S^{N-1}|} \int_{S^{N-1}} e^{\beta N \sum \lambda_i x_i^2} \dd \Omega. 
\eeq
In order to evaluate the integral, we consider 
\beq\label{eq:Jdef}
	J(z) := \int_{\R^N} e^{\beta N \sum \lambda_i y_i^2} e^{-\beta N z \sum y_i^2} \dd \bsy, \qquad z>\lambda_1.
\eeq
We evaluate $J(z)$ the above integral in two different ways. First we evaluate it directly using Gaussian integral and second, we use polar coordinates. By evaluating the Gaussian integrals, we obtain 
\beq\label{eq:JGauss}
	J(z) =  \left(\frac{\pi}{\beta N} \right)^{N/2}  \prod_i \frac{1}{\sqrt{z - \lambda_i}},	\qquad z>\lambda_1.
\eeq
On the other hand, by using polar coordinates, we substitute $\bsy = r \bsx$, $r > 0$, with $\| \bsx \| = 1$ in~\eqref{eq:Jdef}, and then set $\beta N r^2 =t$ to find that 
\beq \label{eq:JLaplace}
	J(z) = \frac1{2(\beta N)^{N/2}}  \int_0^{\infty} e^{-zt} t^{(N/2)-1} I(t) \dd t, 
	\qquad
	I(t):= \int_{S^{N-1}} e^{t \sum \lambda_i x_i^2} \dd \Omega .
\eeq
Note that $J(z)$ is, up to a constant factor, the Laplace transform of $t^{(N/2)-1} I(t)$.
Taking inverse Laplace transform and using~\eqref{eq:JGauss}, we obtain 
\beq \begin{split}
	 \frac{t^{N/2-1}I(t)}{2(\beta N)^{N/2}}  
	 &= \frac{1}{2\pi \ii} \int_{\gamma - \ii \infty}^{\gamma + \ii \infty}  e^{zt} J(z) \dd z 
	 = \left(\frac{\pi}{\beta N} \right)^{N/2}  \frac{1}{2\pi \ii} \int_{\gamma - \ii \infty}^{\gamma + \ii \infty} e^{zt} \prod_i \frac{1}{\sqrt{z - \lambda_i}} \dd z 
\end{split} \eeq
where $\gamma$ is an arbitrary real number satisfying $\gamma>\lambda_1$ since $J(z)$ is defined for $z>\lambda_1$. 
Since $Z_N = \frac{1}{|S^{N-1}|} I \left( \beta N \right)$, we obtain the desired lemma by setting $t=\beta N$ and recalling that $\frac{1}{|S^{N-1}|} = \frac{\Gamma(N/2)}{2\pi^{N/2}}$. 
\end{proof}

From Lemma \ref{inverse laplace}, the partition function~\eqref{partition function'} satisfies 
\beq \label{integral representation}
	Z_N = C_N \int_{\gamma - \ii \infty}^{\gamma + \ii \infty} e^{\frac{N}{2} G(z)} \dd z, 
	\qquad 
	C_N = \frac{\Gamma(N/2)}{2\pi \ii (N\beta)^{N/2-1}},
\eeq
where $\gamma>\lambda_1$, and 
\beq \label{eq:Gdefn}
	G(z) = 2\beta z - \frac{1}{N} \sum_i \log (z - \lambda_i).
\eeq
We use the method of steepest-descent to this integral. 
The following lemma shows that there is a critical value of $G(z)$ on the part of the real line $z\in (\lambda_1, \infty)$ 
and we choose $\gamma$ as this critical point. 

\begin{lem}\label{defofgamma}
There exists a unique $\gamma\in (\lambda_1, \infty)$ satisfying the equation $G'(\gamma) = 0$.
\end{lem}

\begin{proof}
This is immediately obtained by noting that  
\beq \label{eq:G'defn}
G'(z) = 2\beta - \frac{1}{N} \sum_i \frac{1}{z - \lambda_i}
\eeq
is an increasing function of $z \in \mathbb{R}$ on the interval $(\lambda_1, \infty)$ with 
$$
\lim_{z \searrow \lambda_1} G'(z) = -\infty, \qquad \lim_{z \to \infty} G'(z) = 2\beta > 0,
$$
\end{proof}

We also remark that
\beq
G''(z) = \frac{1}{N} \sum_i \frac{1}{(z - \lambda_i)^2} > 0, 
\qquad \text{for $z > \lambda_1$.}
\eeq
Hence $z=\gamma$ is a saddle point of the real part of the function $G(z)$, and $\re (G(z))$ decays fastest along the vertical line $z=\gamma+\ii y$ as $|y|$ increases for small $y$. 

\begin{rem}
The analogue of Lemma \ref{inverse laplace} for Hermitian matrices is the following. 
For a Hermitian matrix $M$ with eigenvalues $\lambda_1\ge\cdots \ge\lambda_N$, we have
\beq \label{Hermitian integral representation}
	\int_{\C S_{N-1}} e^{\beta \langle \bss, M \bss \rangle }\dd \omega_N(\bss) 
	= C_N^H  \int_{\gamma - \ii \infty}^{\gamma + \ii \infty} e^{\frac{N}{2} G_H(z)} \dd z, 
	\quad 
	G_H(z) = \beta z - \frac{1}{N} \sum_i \log (z - \lambda_i),
\eeq
where 
\beq
	C_N^H	= \frac{\Gamma(N)}{2\pi \ii (N\beta)^{N-1}}.
\eeq 
This can be obtained similarly. By diagonalizing $M$ and changing variables, the partition function is 
$$
Z_N = \frac{1}{|S^{2N-1}|} \int_{S^{2N-1}} e^{\beta N \sum \lambda_i (x_{2i-1}^2 + x_{2i}^2)} \dd \Omega.
$$
This is evaluated by considering
$$
J_H(z) := \int_{\R^{2N}} e^{\beta N \sum \lambda_i (y_{2i-1}^2 + y_{2i}^2)} e^{-\beta N z \sum y_i^2} \dd \bsy.
$$
\end{rem}

\section{High temperature case} \label{sec:sub}

In this section, we consider the case $\beta < \beta_c$ and prove Theorem \ref{thm:sub} 
for symmetric ensembles; the proof for Hermitian ensembles can be done in a similar manner by using \eqref{Hermitian integral representation} and we skip its proof.

We use the method of steepest-descent in order to evaluate the integral~\eqref{integral representation} asymptotically. 
Since $G(z)$ is random (since $\lambda_i$ are random), the critical point $\gamma$ from Lemma~\ref{defofgamma} is a random variable. 
We approximate $\gamma$ by a non-random number $\widehat{\gamma}$, which is the critical point 
in the interval $(C_+, \infty)$ of the function
\beq\label{stieltjes gamma3}
	\widehat G(z)= 2\beta z - \int_{C_-}^{C_+} \log (z-k) d \nu(k).
\eeq
The function $\widehat G(z)$ is the version of $G(z)$ in which the random spectral measure is replaced by the limiting non-random spectral measure. 
From the discussions around~\eqref{stieltjes gamma}, if $\beta\in (0, \beta_c)$, there is unique $\widehat{\gamma}$ satisfying 
\beq\label{stieltjes gamma2}
	\widehat G'(\widehat\gamma)= 2\beta - \int_{C_-}^{C_+} \frac{\dd \nu(k)}{\widehat\gamma - k} = 0, \qquad \widehat\gamma \in (C_+, \infty).
\eeq

We start with the following lemma, which holds for all $\beta>0$.

\begin{lem} \label{subcritical gamma0}
Assume Condition~\ref{cond:regular} and Condition~\ref{cond:rigidity}. 
Fix $\delta>0$. 
Then the following hold.
\begin{enumerate}[(i)]
\item For every $\epsilon > 0$,
\beq\label{eq:gammahatgamma0}
	G'(z)-\widehat{G}'(z)= O(N^{-1+\epsilon}) 
\eeq
uniformly in $z\ge C_++\delta$
with high probability.
\item For each $\ell=0,1,2,\cdots$, the derivative $G^{(\ell)}(z)=O(1)$ uniformly in 
$z\in \C\setminus B_\delta$ with high probability where 
$B_\delta=\{x+\ii y: C_--\delta<x<C_++\delta,\,  -\delta<y<\delta\}$.
\end{enumerate}
\end{lem}

\begin{proof}
For (i), set 
\beq \label{sub gamma 200}
	\widetilde G(z)=  2\beta z - \frac{1}{N} \sum_i  \log( z-  \gamma_i )
\eeq
where $\gamma_i$ is the classical location of the $i$-th eigenvalue as defined in~\eqref{eq:classicallocationdef}. 
Then from the rigidity, Condition \ref{cond:rigidity}, 
\beq \label{sub gamma 2}
\left| G'(z) - \widetilde G'(z)  \right|
= \left| \frac{1}{N} \sum_i \frac{(\lambda_i-\gamma_i)}{(z-\lambda_i)(z-\gamma_i)} \right| \leq \frac{C}{N} \sum_i |\lambda_i - \gamma_i| \leq \frac{C N^{\epsilon}}{N}
\eeq
uniformly in $z\ge C_++\delta$ with high probability. 
On the other hand, we claim that 
\beq \label{sub gamma 3}
	\left| \widetilde G'(z) -\widehat G'(z) \right|
	= \left| \frac{1}{N} \sum_i \frac{1}{z - \gamma_i} - \int_{C_-}^{C_+} \frac{\dd \nu(k)}{z - k} \right| \leq \frac{C}{N}
\eeq
uniformly in $z\ge C_++\delta$. 
For this, we define $\widehat \gamma_j$ by
\beq \label{tilde gamma}
\int_{\widehat \gamma_j}^{\infty} \dd \nu(k) = \frac{j}{N}, \qquad j=1,2,\cdots,N,
\eeq
with $\widehat \gamma_0 = C_+$. 
Note that $\widehat \gamma_i \leq \gamma_i \leq \widehat \gamma_{i-1}$. We then have for $i = 2, 3, \cdots, N-1$ that
\beq\label{eq:gahatgaine}
\int_{\widehat\gamma_{i+1}}^{\widehat\gamma_i} \frac{\dd \nu(k)}{z - k} \leq \frac{1}{N} \frac{1}{z - \gamma_i} \leq \int_{\widehat\gamma_{i-1}}^{\widehat\gamma_{i-2}} \frac{\dd \nu(k)}{z - k}
\eeq
for $z\ge C_++\delta$. 
Summing over $i$ and using the trivial estimates 
$\frac{1}{N} \frac{1}{z - \gamma_i} = O(N^{-1})$ and $\int_{\widehat\gamma_i}^{\widehat\gamma_{i-1}} \frac{\dd \nu(k)}{z - k} = O(N^{-1})$,
we find that the desired claim holds. The estimates~\eqref{sub gamma 2} and \eqref{sub gamma 3} imply~\eqref{eq:gammahatgamma0}. 

The part (ii) of the Lemma follows straightforwardly from the formula of $G^{(\ell)}(z)$
and the rigidity, Condition \ref{cond:rigidity}.
\end{proof}

\begin{cor} \label{subcritical gamma}
Assume Condition~\ref{cond:regular} and Condition~\ref{cond:rigidity}.
Let $\beta<\beta_c$. 
Let $\gamma$ be the number defined in Lemma~\ref{defofgamma} 
and let $\widehat{\gamma}$ be defined in~\eqref{stieltjes gamma2}. 
Then for every $\epsilon > 0$,
\beq\label{eq:gammahatgamma}
	\left| \gamma - \widehat\gamma \right| \leq \frac{N^{2\epsilon}}{N}
\eeq
with high probability. 
In particular, there is a constant $c>0$ such that 
\beq\label{eq:gammahatgamma22}
	\gamma-\lambda_1>c 
\eeq
with high probability. 
\end{cor}

\begin{proof}
Since $\widehat\gamma>C_+$, 
choosing $\delta \in (0, (\widehat\gamma -C_+)/2 )$ in Lemma~\ref{subcritical gamma0}, we find that 
\beq \label{sub gamma 1}
 	G'(\widehat\gamma \pm N^{-1+2\epsilon}) = \widehat G'(\widehat\gamma\pm N^{-1+2\epsilon})
	+ O(N^{-1+\epsilon})
\eeq
with high probability.
Now, since $\widehat G'(\widehat \gamma)=0$ and $\widehat G'''(z)=O(1)$ for $z$ near $\widehat\gamma$, the Taylor expansion of $\widehat G$ implies that 
\beq
	 G'(\widehat\gamma \pm  N^{-1+2\epsilon}) 
	= \pm \widehat G''(\widehat\gamma)  N^{-1+2\epsilon}
	 + O(N^{-2+4\epsilon}) + O(N^{-1+\epsilon}).
\eeq
Noting that $\widehat G''(\widehat\gamma)
= \int_{C_-}^{C_+} \frac{\dd \nu(k)}{(\widehat\gamma - k)^2} >0$, this shows that 
\beq
G'(\widehat\gamma - N^{-1+2\epsilon}) < 0, \qquad G'(\widehat\gamma + N^{-1+2\epsilon}) > 0
\eeq
with high probability. 
Since $G'(z)$ is an increasing function of $z$, this proves~\eqref{eq:gammahatgamma}. 

The estimate~\eqref{eq:gammahatgamma22} is a consequence of~\eqref{eq:gammahatgamma}, ~\eqref{rigidity}, ~\eqref{eq:classicalfromC}, and the fact that $\widehat\gamma$ is a non-random number, independent of $N$, satisfying $\widehat\gamma > C_+$.
\end{proof}

\begin{cor} \label{cor:G gamma}
Assume Condition~\ref{cond:regular} and Condition~\ref{cond:rigidity} and let $\beta<\beta_c$. 
Then for every $\epsilon > 0$
\beq\label{eq:Ggammacor}
G(\gamma) = G(\widehat\gamma) + O(N^{-2+4\epsilon}), \qquad G''(\gamma) = G''(\widehat\gamma) + O(N^{-1+2\epsilon})
\eeq
with high probability. 
\end{cor}

\begin{proof}
From the Taylor expansion, the definition of $\gamma$, and Lemma~\ref{subcritical gamma0} (ii), 
\beq
	G(\widehat\gamma) = G(\gamma) + G'(\gamma) (\widehat\gamma - \gamma) + O(|\widehat\gamma - \gamma|^2)
	= G(\gamma) + O(|\widehat\gamma - \gamma|^2)
\eeq
and
\beq
G''(\gamma) = G''(\widehat\gamma) + O(|\gamma - \widehat\gamma|).
\eeq
The estimate~\eqref{eq:Ggammacor} now follows from Corollary~\ref{subcritical gamma}. 
\end{proof}

We now evaluate the integral~\eqref{integral representation} using the method of steepest-descent.

\begin{lem} \label{steepest descent sub}
Assume Condition~\ref{cond:regular}  and Condition~\ref{cond:rigidity} and let $\beta<\beta_c$. 
Then for every $\epsilon > 0$,  
\beq
	 \int_{\gamma - \ii \infty}^{\gamma + \ii \infty} e^{\frac{N}{2} G(z)} \dd z = \ii e^{\frac{N}{2} G(\widehat\gamma)} \sqrt{\frac{4\pi}{N G''(\widehat\gamma)}} \left( 1 + O(N^{-1+6\epsilon}) \right)
\eeq
with high probability. 
\end{lem}

\begin{proof}
We had chosen $\gamma$ as the critical point of $G(z)$ such that $\gamma>\lambda_1$. 
For this proof, it is enough to use the straight line $\gamma+\ii\R$ for the contour instead of the path of steepest-descent. 
Changing the variables, we have 
\beq \begin{split} \label{steep sub}
\int_{\gamma - \ii \infty}^{\gamma + \ii \infty} e^{\frac{N}{2} G(z)} \dd z &= \frac{\ii}{\sqrt N} \int_{-\infty}^{\infty} \exp \left[ \frac{N}{2} G \big(\gamma + \ii \frac{t}{\sqrt N} \big) \right] \dd t \\
&= \frac{\ii e^{\frac{N}{2} G(\gamma)}}{\sqrt N} \int_{-\infty}^{\infty} \exp \left[ \frac{N}{2} \left( G \big(\gamma + \ii \frac{t}{\sqrt N} \big) - G(\gamma) \right) \right] \dd t.
\end{split} \eeq
We now estimate the integral in the right hand side of \eqref{steep sub}. 
First, we have
\beq \begin{split}
&\int_{-N^{\epsilon}}^{N^{\epsilon}} \exp \left[ \frac{N}{2} \left( G \big(\gamma + \ii \frac{t}{\sqrt N} \big) - G(\gamma) \right) \right] \dd t \\
&= \int_{-N^{\epsilon}}^{N^{\epsilon}} \exp \left[ \frac{N}{2} \left( \frac{G''(\gamma)}{2} \frac{(\ii t)^2}{N} + \frac{G'''(\gamma)}{6} \frac{(\ii t)^3}{N^{3/2}} + O(N^{-2+4\epsilon}) \right) \right] \dd t \\
&= \int_{-N^{\epsilon}}^{N^{\epsilon}} \exp \left[ -\frac{G''(\gamma)}{4} t^2 \right] \dd t - \ii \int_{-N^{\epsilon}}^{N^{\epsilon}} \frac{G'''(\gamma)}{12} \frac{t^3}{\sqrt N} \exp \left[ -\frac{G''(\gamma)}{4} t^2 \right] \dd t + O(N^{-1+6\epsilon})
\end{split} \eeq
with high probability, 
where we used that $G''(\gamma) > 0$ and $G^{(\ell)}(\gamma + \ii t) = O(1)$ for any $t \in \R$ and $\ell=3,4$ (see Lemma~\ref{subcritical gamma0} (ii)). 
The integral in the middle vanishes since the integrand is an odd function of $t$. 
From the estimate $\int_{N^{\epsilon}}^{\infty} e^{-t^2} \dd t = O(N^{-\epsilon}e^{-N^{2\epsilon}})$, 
we obtain that
\beq\label{eq:steepdescep}
\int_{-N^{\epsilon}}^{N^{\epsilon}} \exp \left[ \frac{N}{2} \left( G \big(\gamma + \ii \frac{t}{\sqrt N} \big) - G(\gamma) \right) \right] \dd t = \sqrt{\frac{4\pi}{G''(\gamma)}} \left( 1 + O(N^{-1+6\epsilon}) \right).
\eeq
Second, the tail part of the integral in the right hand side of \eqref{steep sub} satisfies
\beq \begin{split} \label{eq:tailestimate}
&\left| \int_{N^{\epsilon}}^{\infty} \exp \left[ \frac{N}{2} \left( G \big(\gamma + \ii \frac{t}{\sqrt N} \big) - G(\gamma) \right) \right] \dd t \right| \\
&\leq \int_{N^{\epsilon}}^{\infty} \left| \exp \left[ -\frac{1}{2} \sum_{j=1}^N \log \left( \frac{\gamma - \lambda_j + \ii t N^{-1/2}}{\gamma - \lambda_j} \right) \right] \right| \dd t \\
&\leq \int_{N^{\epsilon}}^{\infty} \exp \left[ -\frac{N}{4} \log \left( 1 + \frac{c t^2}{N} \right) \right] \dd t \\
&\leq \int_{N^{\epsilon}}^N e^{-\frac{c}{8}N^{2\epsilon}} \dd t + \int_N^{\infty} (cN^{-1} t^2)^{-N/4} \dd t 
=O(e^{-N^{\epsilon}}) + O(N^{-N/8})
\end{split} \eeq
with high probability, where we used~\eqref{eq:gammahatgamma22}. 
Thus, the tail part is negligible, and hence we obtain from~\eqref{steep sub} that
\beq
\int_{\gamma - \ii \infty}^{\gamma + \ii \infty} e^{\frac{N}{2} G(z)} \dd z = \frac{\ii e^{\frac{N}{2} G(\gamma)}}{\sqrt N} \sqrt{\frac{4\pi}{G''(\gamma)}} \left( 1 + O(N^{-1+6\epsilon}) \right).
\eeq

Finally, using Corollary \ref{cor:G gamma}, we conclude that
\beq
\int_{\gamma - \ii \infty}^{\gamma + \ii \infty} e^{\frac{N}{2} G(z)} \dd z = \frac{\ii e^{\frac{N}{2} G(\widehat\gamma)}}{\sqrt N} \sqrt{\frac{4\pi}{G''(\widehat\gamma)}} \left( 1 + O(N^{-1+6\epsilon}) \right)
\eeq
with high probability.
\end{proof}

We now prove Theorem \ref{thm:sub}.

\begin{proof}[Proof of Theorem \ref{thm:sub}]
From Lemmas \ref{inverse laplace} and \ref{steepest descent sub}, for every $\epsilon>0$, 
\beq
	Z_N = \ii C_N e^{\frac{N}{2} G(\widehat\gamma)} \sqrt{\frac{4\pi}{N G''(\widehat\gamma)}} \left( 1 + O(N^{-1+\epsilon}) \right)
\eeq
with high probability. 
Since 
\beq \label{n sphere}
	C_N = \frac{\Gamma(N/2)}{2\pi \ii (N\beta)^{N/2-1}}
	=\frac{\sqrt{N}\beta}{\ii \sqrt{\pi} (2\beta e)^{N/2}} \left( 1 + O(N^{-1}) \right)
\eeq
from the Stirling's formula, we find that 
\beq \begin{split}
	Z_N
	&=e^{\frac{N}{2} G(\widehat\gamma) } \frac{2\beta }{ (2\beta e)^{N/2} \sqrt{G''(\widehat\gamma)}} \left( 1 + O(N^{-1+\epsilon}) \right)
\end{split} \eeq
and hence
\beq \label{eq:Ftempsub} \begin{split}
F_N &= \frac{1}{N} \log Z_N 
= \frac{1}{2} \left( G(\widehat\gamma) -1 - \log (2\beta)  \right) 
+ \frac1{N} \left( \log(2\beta) -\frac12  \log G''(\widehat\gamma)\right)  + O(N^{-2+\epsilon}) \\
\end{split} \eeq
with high probability. 

Define the functions
\beq
\varphi(x) := \log (\widehat\gamma - x), \qquad \psi(x):= \frac1{(\widehat\gamma -x)^2}
\eeq
for $x\in [C_+, C_-]$ and extend $\varphi$ and $\psi$ to  bounded $C^\infty$ functions with compact support on the real line. 
We then have
\beq
G(\widehat\gamma) = 2\beta \widehat\gamma - \frac{1}{N} \sum_i \varphi(\lambda_i), 
\qquad G''(\widehat\gamma) = \frac1{N} \sum_i \psi(\lambda_i)
\eeq
with high probability. 
Regarding $\widehat\gamma$ as a function of $\beta$, set
\beq \label{eq:f_0}
f_0(\beta) := \int_{C_-}^{C_+} \log (\widehat\gamma - k) \dd \nu(k),
\qquad f_2(\beta) := \int_{C_-}^{C_+} \frac{\dd \nu(k)}{(\widehat\gamma - k)^2}.
\eeq
We find from Condition \ref{cond:linear} that 
\beq
	\caN_{\varphi} := \sum_i \varphi(\lambda_i) - N \int_{C_-}^{C_+} \varphi(k) \dd \nu(k)
=\sum_i \varphi(\lambda_i) - N f_0(\beta) 
\eeq
converges in distribution to Gaussian random variable with mean $\mean(\varphi)$ and variance $V(\varphi)$. 
Similarly, we also have that
\beq
	\caN_{\psi} :=\sum_i \psi(\lambda_i) - N\int_{C_-}^{C_+} \psi(k) \dd \nu(k)
	= \sum_i \psi(\lambda_i) - N f_2(\beta) 
\eeq
converges in distribution to Gaussian random variable with mean $\mean(\psi)$ and variance $V(\psi)$.
 
Thus,~\eqref{eq:Ftempsub} becomes
\beq \begin{split} \label{eq:Ftempsub2}
F_N &= \left( \beta \widehat\gamma - \frac{f_0(\beta)+1 + \log (2\beta)}{2} \right) - \frac1{2N} \caN_{\varphi} \\
&\qquad +  \frac{2 \log (2\beta) - \log ( f_2(\beta))}{2N} - \frac{1}{2N} \log \left( 1 + \frac{\caN_{\psi}}{N f_2(\beta)} \right) + O(N^{-1+\epsilon})
\end{split} \eeq
with high probability. Since $\caN_{\psi}$ converges to a Gaussian random variable, we see, in particular, that the random variable $\log ( 1 + \frac{\caN_{\psi}}{N f_2(\beta)} )$ in the right hand side of \eqref{eq:Ftempsub2} converges in distribution to $0$ as $N \to \infty$. Since the convergence in distribution to the constant $0$ implies the convergence in probability to $0$, using Slutsky theorem, we conclude that
$$
N F_N - N \left( \beta \widehat\gamma - \frac{f_0(\beta)+1 + \log (2\beta)}{2} \right) -  \frac{2 \log (2\beta) - \log ( f_2(\beta))}{2}
$$
converges in distribution to a Gaussian with mean $-\frac1{2} \mean(\varphi)$ and variance $\frac1{4} V(\varphi)$. This completes the proof of the theorem.

\end{proof}

\section{Low temperature case} \label{sec:sup}

In this section, we consider the case $\beta > \beta_c$ and prove Theorem \ref{thm:sup}. 
As in the last section, we only give a proof for symmetric random matrices; the proof for Hermitian random matrices is similar. 

As we saw in the previous section, the location of $\gamma$, the critical point of the function $G$, is crucial in the asymptotic evaluation of the integral in~\eqref{integral representation}. 
When $\beta<\beta_c$ we approximated $\gamma$ by the deterministic number $\widehat\gamma$ 
that is the critical point of the deterministic approximation $\widehat G(z)$ of the function $G(z)$. 
However, when $\beta>\beta_c$, $\widehat G(z)$ does not have any critical point in $z>C_+$ and we cannot approximate $\gamma$ by a deterministic number. 
The following lemma shows that $\gamma$ is close to $\lambda_1$ up to order about $1/N$.

\begin{lem} \label{supercritical gamma}
Assume Conditions \ref{cond:regular} and \ref{cond:rigidity}.
Let $\beta > \beta_c$. 
Then for every $0<\epsilon <\frac14$,
\beq
\frac{1}{3\beta N} \leq \gamma - \lambda_1 \leq \frac{N^{4\epsilon}}{N}
\eeq
with high probability. 
\end{lem}

\begin{proof}
Since 
\beq
	G'(z) =2\beta - \frac1{N} \sum_i \frac{1}{z-\lambda_i} < 2\beta - \frac1{N} \frac{1}{z-\lambda_1}, \qquad z>\lambda_1,
\eeq
we see that $G'(\lambda_1 + \frac{1}{3\beta N}) < 0$.

Since $G'(z)$ is an increasing function for $z>\lambda_1$, the Lemma is proved if we show that $G'(\lambda_1 + N^{-1 + 4\epsilon}) > 0$ with high probability. 
By Condition \ref{cond:rigidity}, we may assume that the eigenvalues $\lambda_i$'s satisfy the rigidity~\eqref{rigidity}: this event occurs with high probability. 
For such $\lambda_i$'s we need to show that  $G'(\lambda_1 + N^{-1 + 4\epsilon}) > 0$. 
In order to show this, we write 
\beq
G'(z) = 2\beta - \frac{1}{N} \sum_{i=1}^{N^{3\epsilon}} \frac{1}{z - \lambda_i} - \frac{1}{N} \sum_{i=N^{3\epsilon}+1}^{N-N^{3\epsilon}} \frac{1}{z - \lambda_i} - \frac{1}{N} \sum_{i=N-N^{3\epsilon}+1}^{N} \frac{1}{z - \lambda_i}
\eeq
with $z=\lambda_1 + N^{-1 + 4\epsilon}$. 
For $1 \leq i \leq N^{3\epsilon}$, since $\lambda_1\ge \lambda_i$,
\beq
\frac{1}{N} \sum_{i=1}^{N^{3\epsilon}} \frac{1}{\lambda_1 + N^{-1 + 4\epsilon} - \lambda_i} = O(N^{-\epsilon}).
\eeq
For $N^{3\epsilon} < i \leq N-N^{3\epsilon}$, we have from the rigidity condition \eqref{rigidity} that
$|\lambda_i - \gamma_i| \leq N^{-2/3}$. 
We also note that $C_+ -\gamma_i \geq C^{-1}N^{-2/3 + 2\epsilon}$ from~\eqref{eq:classicalfromC}. 
On the other hand, since $\lambda_1-\gamma_1=O(N^{-2/3+\epsilon})$ from~\eqref{rigidity} and $C_+-\gamma_1=O(N^{-2/3})$ from~\eqref{eq:classicalfromC}, 
we have $\lambda_1=C_++O(N^{-2/3+\epsilon})$. 
Therefore we find that
\beq
\frac{1}{N} \sum_{i=N^{3\epsilon}+1}^{N-N^{3\epsilon}} \frac{1}{\lambda_1 + N^{-1 + 4\epsilon} - \lambda_i} = \frac{1}{N} \sum_{i=N^{3\epsilon}+1}^{N-N^{3\epsilon}} \frac{1}{C_+ - \gamma_i} \left( 1 + O(N^{-\epsilon}) \right).
\eeq
The right hand side can be estimated by applying the idea used in the proof of Lemma \ref{subcritical gamma0}. Recall the definition of $\widehat \gamma_i$ in \eqref{tilde gamma}. Summing the inequalities (see~\eqref{eq:gahatgaine})
\beq \label{sum_integral 1}
\int_{\widehat\gamma_{i+1}}^{\widehat\gamma_i} \frac{\dd \nu(k)}{C_+ - k} \leq \frac{1}{N} \frac{1}{C_+ - \gamma_i} \leq \int_{\widehat\gamma_{i-1}}^{\widehat\gamma_{i-2}} \frac{\dd \nu(k)}{C_+ - k}
\eeq
over $i$ from $N^{3\epsilon}+1$ to $N-N^{3\epsilon}$, 
and recalling that $\beta_c=\frac12 \int_{C_-}^{C_+} \frac{\dd \nu(k)}{C_+ - k}$ by Definition~\ref{def:beta_c},  
we find that
\beq \begin{split} \label{sum_integral 2}
	&\left| \frac{1}{N} \sum_{i=N^{3\epsilon}+1}^{N-N^{3\epsilon}} \frac{1}{C_+ - \gamma_i} -2\beta_c  \right| 
	= \left| \frac{1}{N} \sum_{i=N^{3\epsilon}+1}^{N-N^{3\epsilon}} \frac{1}{C_+ - \gamma_i} - \int_{C_-}^{C_+} \frac{\dd \nu(k)}{C_+ - k} \right| \\
&\leq C \left( \int_{C_-}^{\widehat\gamma_{N-N^{3\epsilon} -2}} \frac{\dd \nu(k)}{C_+ - k} + \int_{\widehat\gamma_{N^{3\epsilon} +2}}^{C_+} \frac{\dd \nu(k)}{C_+ - k} \right) = O(N^{-1/3 + \epsilon}).
\end{split} \eeq
Finally, for $N-N^{3\epsilon} < i \leq N$, since $\lambda_1- \lambda_i\ge 1$,
we find 
\beq
\frac{1}{N} \sum_{i=N-N^{3\epsilon}+1}^{N} \frac{1}{\lambda_1 + N^{-1 + 4\epsilon} - \lambda_i} = O(N^{-1 + 3\epsilon}).
\eeq 
Combining the estimates, we find that 
\beq
G'(\lambda_1 + N^{-1 + 4\epsilon}) = 2\beta - 2\beta_c + o(1) > 0.
\eeq
This proves the lemma.
\end{proof}

We will show in Lemma \ref{steepest descent sup} below that the method of steepest-descent still applies and the main contribution to the integral representation of $Z_N$ comes from $G(\gamma)$. 
The value of $G(\gamma)$ is, heuristically, 
\beq
\begin{split}
	G(\gamma) 
	&= 2\beta \gamma - \frac1{N} \sum_{i=1}^N \log (\gamma-\lambda_i) \\
	&\approx 2\beta \gamma - \frac1{N} \sum_{i=1}^N \left[ \log(C_+- \lambda_i) + \frac1{C_+-\lambda_i}
	(\gamma-C_+) \right]\\
	&\approx 2\beta\gamma - \int_{C_-}^{C_+} \log (C_+-z)d\nu(z) - \int_{C_-}^{C_+} \frac{d\nu(z)} {C_+-z} (\gamma-C_+)\\
	&\approx 2\beta\lambda_1 - \int_{C_-}^{C_+} \log (C_+-z)d\nu(z) - 2\beta_c (\lambda_1-C_+)
\end{split}
\eeq
where we used~\eqref{eq:defbetac} and Lemma \ref{supercritical gamma} for the last line.
We show this approximation rigorously and also estimate the derivatives of $G(\gamma)$ in the next Lemma.

\begin{lem} \label{G gamma bound}
Assume Conditions \ref{cond:regular} and \ref{cond:rigidity}.  
Let $\beta > \beta_c$. 
Let $\gamma$ be the solution of the equation $G'(\gamma) = 0$ in Lemma \ref{supercritical gamma}. Then 
for $0<\epsilon<\frac14$
\beq
	G(\gamma) = 2\beta \lambda_1 - \int_{C_-}^{C_+} \log (C_+-z)d\nu(z) 
	- 2\beta_c (\lambda_1 - C_+) + O(N^{-1+4\epsilon})
\eeq
with high probability. 
Moreover, for $0<\epsilon<\frac14$, there is a constant $C_0>0$ such that 
\beq
	N^{\ell -1 - 4\ell \epsilon} \leq  \frac{(-1)^{\ell}}{(\ell-1)!} G^{(\ell)}(\gamma) \le C_0^\ell N^{\ell -1 +3\epsilon}
\eeq
for all $\ell = 2, 3, \cdots$
with high probability. Here, $C_0$ does not depend on $\ell$. 
\end{lem}

\begin{proof}
We may assume that the eigenvalues $\lambda_i$'s satisfy the rigidity Condition \ref{cond:rigidity} and Lemma~\ref{supercritical gamma} since these event occurs with high probability. 

Since we have from Lemma \ref{supercritical gamma} that $\gamma = \lambda_1 + O(N^{-1+4\epsilon})$, the first part of the lemma is proved if we show that 
\begin{equation}\label{eq:firstpartof}
\left| \frac{1}{N} \sum_{i=1}^N \log (\gamma - \lambda_i) - \int_{C_-}^{C_+} \log (C_+ - k) \dd \nu(k) - 2\beta_c (\gamma - C_+) \right| = O(N^{-1+4\epsilon}).
\end{equation}
Mimicking the proof of Lemma \ref{supercritical gamma}, we first consider $1 \leq i \leq N^{3\epsilon}$.
Since $\gamma-\lambda_1\ge \frac1{3\beta N}$ with high probability by Lemma~\ref{supercritical gamma},  we have the trivial estimate
\beq \label{log estimate case 1}
\left| \frac{1}{N} \sum_{i=1}^{N^{3\epsilon}} \log (\gamma - \lambda_i) \right| \leq \frac{C N^{3\epsilon}}{N} \log N.
\eeq
Similarly, 
\beq \label{log estimate case 2}
\left| \frac{1}{N} \sum_{i=N-N^{3\epsilon}+1}^{N} \log (\gamma - \lambda_i) \right| \leq \frac{C N^{3\epsilon}}{N}.
\eeq

We now consider the case $N^{3\epsilon} < i < N-N^{3\epsilon}$. 
Note that
\beq \label{eq:log gamma in G}
\begin{split}
	\log (\gamma - \lambda_i) - \log (C_+ -\gamma_i) 
	&= \log\left( 1 + \frac{\gamma-C_+}{C_+-\gamma_i} + \frac{\gamma_i-\lambda_i}{C_+-\gamma_i}\right) \\
	&= \frac{\gamma - C_+}{C_+ -\gamma_i} + O \left( \left( \frac{\gamma - C_+}{C_+ -\gamma_i} \right)^2 \right) + O \left( \frac{\lambda_i - \gamma_i}{C_+ -\gamma_i} \right).
\end{split}
\eeq
The first error term can be estimated by
\beq
	\left( \frac{\gamma - C_+}{C_+ -\gamma_i} \right)^2 \leq C \frac{N^{-4/3 + 2\epsilon}}{i^{4/3} N^{-4/3}} = C N^{2\epsilon} i^{-4/3}
\eeq
from~\eqref{eq:classicalfromC} and $\gamma-C_+=(\gamma-\lambda_1)+(\lambda_1-C_+)=O(N^{-2/3 + \epsilon})$ due to Lemma~\ref{supercritical gamma}. 
For the second error term in \eqref{eq:log gamma in G}, we consider two different cases. For $N^{3\epsilon} < i \leq N/2$, we use the estimate
\beq
\frac{|\lambda_i - \gamma_i|}{C_+ -\gamma_i} \leq C \frac{i^{-1/3} N^{-2/3+\epsilon}}{i^{2/3} N^{-2/3}} = C N^{\epsilon} i^{-1}
\eeq
from~\eqref{eq:classicallocationdef} and~\eqref{eq:classicalfromC}. 
For $N/2 < i < N-N^{3\epsilon}$, we simply use
\beq
\frac{|\lambda_i - \gamma_i|}{C_+ -\gamma_i} \leq C N^{-1 + \epsilon}.
\eeq
Since 
\beq
	\frac{1}{N} \sum_{i=N^{3\epsilon}+1}^{N-N^{3\epsilon}}  N^{2\epsilon} i^{-4/3} 
	\leq CN^{-1+\epsilon}, 
	\qquad
	\frac{1}{N} \sum_{i=N^{3\epsilon}+1}^{N/2}  N^{\epsilon} i^{-1}
	\leq CN^{-1+\epsilon} \log N, 
\eeq
we obtain, after summing~\eqref{eq:log gamma in G} over $i$,  that
\beq \label{log estimate case 3}
\left| \frac{1}{N} \sum_{i=N^{3\epsilon}+1}^{N-N^{3\epsilon}} \log (\gamma - \lambda_i) - \frac{1}{N} \sum_{i=N^{3\epsilon}+1}^{N-N^{3\epsilon}} \log (C_+ - \gamma_i) - \frac{\gamma - C_+}{N} \sum_{i=N^{3\epsilon}+1}^{N-N^{3\epsilon}} \frac{1}{C_+ -\gamma_i} \right| \leq \frac{C N^{\epsilon} \log N}{N}.
\eeq

From \eqref{log estimate case 1}, \eqref{log estimate case 2}, and \eqref{log estimate case 3}, we conclude that
\beq \label{log estimate case 4}
\left| \frac{1}{N} \sum_{i=1}^N \log (\gamma - \lambda_i) - \frac{1}{N} \sum_{i=N^{3\epsilon}+1}^{N-N^{3\epsilon}} \log (C_+ - \gamma_i) - \frac{\gamma - C_+}{N} \sum_{i=N^{3\epsilon}+1}^{N-N^{3\epsilon}} \frac{1}{C_+ -\gamma_i} \right| = O(N^{-1+4\epsilon}).
\eeq
We proved in~\eqref{sum_integral 2} that
\beq
\left| \frac{1}{N} \sum_{i=N^{3\epsilon}+1}^{N-N^{3\epsilon}} \frac{1}{C_+ - \gamma_i} -2\beta_c  \right| = O(N^{-1/3 + \epsilon}).
\eeq
On the other hand, following the arguments in \eqref{sum_integral 1} and \eqref{sum_integral 2}, we obtain 
\beq \begin{split}
&\left| \frac{1}{N} \sum_{i=N^{3\epsilon}+1}^{N-N^{3\epsilon}} \log (C_+ -\gamma_i) - \int_{C_-}^{C_+} \log (C_+ - k) \dd \nu(k) \right|\\
&\leq C \left( \int_{C_-}^{\widehat\gamma_{N-N^{3\epsilon} -2}} \log (C_+ - k) \dd \nu(k) + \int_{\widehat\gamma_{N^{3\epsilon} +2}}^{C_+} \log (C_+ - k) \dd \nu(k) \right) \\
&\leq C N^{-1+3\epsilon} + C \int_0^{N^{-2/3+2\epsilon}} \sqrt\kappa \log \kappa\, \dd \kappa \\
& \leq C N^{-1+3\epsilon} \log N.
\end{split} \eeq
Inserting these estimates into~\eqref{log estimate case 4}, we obtain~\eqref{eq:firstpartof}. 
Hence the first part of the lemma is proved.

We now prove the second part of the lemma. We have
\beq
	G^{(\ell)}(\gamma) = \frac{(-1)^{\ell} (\ell-1)!}{N} \sum_{i=1}^N \frac1{(\gamma-\lambda_i)^{\ell}}
\eeq
for $\ell=2,3,\cdots$. 
For the lower bound, we use Lemma \ref{supercritical gamma} to obtain
\beq
\frac{1}{N} \sum_{i=1}^N \frac{1}{(\gamma - \lambda_i)^{\ell}} \geq \frac{1}{N} \frac{1}{(\gamma - \lambda_1)^{\ell}} \geq N^{\ell-1 -4\ell\epsilon}.
\eeq
For the upper bound, we use 
\beq \label{eq:Lemma sup two10}
\frac{1}{N} \sum_{i=1}^{N^{3\epsilon}} \frac{1}{(\gamma - \lambda_i)^{\ell}} \leq N^{-1+3\epsilon} \frac{1}{(\gamma - \lambda_1)^\ell} \leq (3\beta)^\ell N^{\ell-1+3\epsilon}
\eeq
and
\beq
\frac{1}{N} \sum_{i=N^{3\epsilon}+1}^N \frac{1}{(\gamma - \lambda_i)^\ell} \leq \frac{1}{N} \sum_{i=N^{3\epsilon}+1}^N \frac{2^\ell}{(C_+ - \gamma_i)^\ell} \leq \frac{2^{\ell}}{N} \sum_{i=N^{3\epsilon}+1}^N \frac{1}{(i^{2/3} N^{-2/3})^\ell} \leq 3\cdot 2^{\ell} N^{(\frac23\ell-1)(1-3 \epsilon)}.
\eeq
This completes the proof of the lemma.

\end{proof}

In the supercritical case, we have the following lemma corresponding to Lemma \ref{steepest descent sub} in the subcritical case.

\begin{lem} \label{steepest descent sup}
Assume Conditions \ref{cond:regular} and \ref{cond:rigidity}.  
Let $\beta > \beta_c$. 
Let $\gamma$ be the unique solution to the equation $G'(\gamma) = 0$ defined in Lemma \ref{defofgamma}. Then, there exists $K \equiv K(N)$ satisfying $N^{-C} < K < C$ for some constant $C > 0$ such that
\beq \label{eq:K}
	\int_{\gamma - \ii \infty}^{\gamma + \ii \infty} e^{\frac{N}{2} G(z)} \dd z = \ii e^{\frac{N}{2} G(\gamma)} K
\eeq
with high probability. 
\end{lem}

\begin{proof}
Define $K$ by the relation~\eqref{eq:K}:
\beq
	K:= -\ii e^{-\frac{N}{2} G(\gamma)} \int_{\gamma - \ii \infty}^{\gamma + \ii \infty} e^{\frac{N}{2} G(z)} \dd z.
\eeq
Then $K$ is real and positive since $Z > 0$ in \eqref{integral representation} and $G(\gamma)$ is real.  
We now estimate $K$. 

Fix $0<\epsilon<\frac1{26}$. 
Since $\epsilon<\frac1{4}$,  
we may assume that, as in the previous two lemmas, the eigenvalues $\lambda_i$'s satisfy the rigidity Condition \ref{cond:rigidity}, Lemma~\ref{supercritical gamma} and Lemma~\ref{G gamma bound} since these event occurs with high probability. 

To prove the upper bound for $K$, we consider
\beq \begin{split}
K &= -\ii \int_{\gamma - \ii \infty}^{\gamma + \ii \infty} \exp \left( \frac{N}{2} \left[ G(z) - G(\gamma) \right] \right) \dd z = \int_{-\infty}^{\infty} \exp \left( \frac{N}{2} \left[ G(\gamma + \ii t) - G(\gamma) \right] \right) \dd t \\
&= \int_{-\infty}^{\infty} \exp \left( \ii N \beta t - \frac{1}{2} \sum_{j=1}^N \log \left( 1 + \frac{\ii t}{\gamma - \lambda_j} \right) \right) \dd t \\
&\leq \int_{-\infty}^{\infty} \exp \left( - \frac{1}{4} \sum_{j=1}^N \log \left( 1 + \frac{t^2}{(\gamma - \lambda_j)^2} \right)\right) \dd t
\end{split} \eeq
where the last step was obtained by taking the absolute value of the integrand. 
Since $\gamma-\lambda_i\le \gamma-\lambda_N \le C$ for some $C$ (with high probability), we find that 
\beq \begin{split}
	K&\leq \int_{-\infty}^{\infty} \exp \left( - \frac{N}{4} \log \left( 1 + C t^2 \right) \right) \dd t = \int_{-\infty}^{\infty} (1 + C t^2)^{-N/4} \dd t \leq C
\end{split} \eeq
for some $C > 0$, if $N$ is sufficiently large. This proves the upper bound for $K$.

In order to prove the lower bound for $K$, we consider the curve of steepest-descent that passes through the point $\gamma$ in the complex plane. 
This curve, denoted by $\Gamma$, satisfies $\im G(z) = 0$. (Note that the real axis also satisfies $\im G(z) = 0$.)  
It is easy to check, using the formula~\eqref{eq:Gdefn} of $G(z)$, that 

(i) $\Gamma\cap \C_+$ is a $C^1$ curve,

(ii) $\Gamma$ intersects the real axis only at $\gamma$,

(iii) the tangent line of $\Gamma$ at $\gamma$ is parallel to the imaginary axis, and 

(iv) the real axis is the asymptote in the $-\infty$ direction. 

Since all solutions of the equation $G'(z) = 0$ are on the real axis, the function $G(z)$ is decreasing along the curve $\Gamma \cap \C^+$ as $z$ moves from the point $\gamma$ to the point $-\infty$. 

We first have 
\beq \label{contour deform}
	\int_{\gamma - \ii \infty}^{\gamma + \ii \infty} e^{\frac{N}{2} G(z)} \dd z
	= \int_\Gamma  e^{\frac{N}{2} G(z)}  \dd z.
\eeq
This follows by noting that 
$\re G(z) \leq 2\beta \gamma - \log (R/2)$, for $|z|=R$ such that $\re z\le \gamma$, and hence, if we let $C_R$ be the circular arc $|z|=R$ such that $\re z\le \gamma$, 
then 
\beq 
	\left|  \int_{C_R}  e^{\frac{N}{2} G(z)}  \dd z \right|
	\le  \frac{C_N}{R^{N/4-1}} \to 0 
\eeq
as $R\to\infty$ where $C_N$ is a constant depending on $N$. 
From~\eqref{contour deform}, we have
\beq \label{eq:Knew form}
\begin{split}
	K &= -\ii \int_{\Gamma} \exp \left( \frac{N}{2} \left[ G(z) - G(\gamma) \right] \right) \dd z . 
\end{split} \eeq

Let
\beq
\Gamma^+ = \Gamma \cap \C^+, \qquad \Gamma^- = \Gamma \cap \C^-.
\eeq
If substitute $\dd z = \dd x + \ii \dd y$ in~\eqref{eq:Knew form}, we see that the integral with respect to $\dd x$ vanishes since $K>0$. (More precisely, the contribution from $\exp [(N/2) (G(z) - G(\gamma)] (-\dd x)$ on $\Gamma^+$ is cancelled by the contribution from $\exp [(N/2) (G(z) - G(\gamma)] \dd x$ on $\Gamma^-$.) 
We thus find, using that $\overline{G(\bar{z}})=G(z)$, that 
\beq \label{eq:define K}
K = 2 \int_{\Gamma^+} \exp \left( \frac{N}{2} \left[ G(z) - G(\gamma) \right] \right) \dd y.
\eeq

Now let $B_{N^{-2}}$ be the ball of radius $N^{-2}$ centered at $\gamma$. 
Since $G(z)$ is analytic for $z$ such that $\re z>\lambda_1$, $G(z)$ is analytic in $\overline{B}_{N^{-2}}$ from Lemma~\ref{supercritical gamma}. 
Hence for $z \in \Gamma^+ \cap B_{N^{-2}}$,
\beq \label{eq:G Taylor}
	G(z) - G(\gamma) = \sum_{j=2}^{\infty} \frac{G^{(j)}(\gamma)}{j!} (z-\gamma)^j.
\eeq
Due to the second part of Lemma  \ref{G gamma bound}, this power series converges uniformly in $\overline{B}_{N^{-2}}$ for all $N>C_0$ where $C_0$ is the constant in Lemma  \ref{G gamma bound}.  
Recall that $G^{(j)}(\gamma)$ is a real number for all $j$; more precisely, $(-1)^j G^{(j)}(\gamma)>0$. 
Comparing the imaginary parts of the both sides of~\eqref{eq:G Taylor}, we find that
for $z\in \Gamma_+\cap \overline{B}_{N^{-2}}$, 
\beq \label{im on Gamma}
0 = G''(\gamma)  \re (z-\gamma) \im z + \frac{G'''(\gamma)}{2} \big( \re (z-\gamma) \big)^2 \im z - \frac{G'''(\gamma)}{6} (\im z)^3 + \widetilde\Omega
\eeq
with
\beq
	\widetilde\Omega = \sum_{j=4}^\infty \frac{G^{(j)}(\gamma)}{j!} \im \left( (z-\gamma)^j \right).
\eeq
Note that $\im \left( (z-\gamma)^j \right)$ divided by $\im (z-\gamma)$ is a polynomial in $\re (z-\gamma)$ and $\im z$. 
Dividing~\eqref{im on Gamma} by $\im z$ and $G''(\gamma)$, we find that 
$z=(X+\gamma) + \ii Y \in \Gamma_+\cap \overline{B}_{N^{-2}}$ solves the equation 
\beq \label{im on Gamma2}
	F(X, Y)=0, 
	\qquad 
	F(X, Y) :=  X - \frac{\alpha}2 X^2 + \frac{\alpha}6 Y^2 + \Omega(X,Y) 
\eeq
where
\beq
	\alpha:= - \frac{G'''(\gamma)}{G''(\gamma)} >0, 
\eeq
and 
\beq
	\Omega(X,Y) = \sum_{j=4}^\infty \frac{G^{(j)}(\gamma)}{j! G''(\gamma)} \frac{\im \left( (X+\ii Y)^j \right) }{Y}.
\eeq
Using the general inequality
\beq \label{eq:xplusyjima}
	|\im \left( (a+\ii b)^j\right)| \leq  j|a+\ii b|^{j-1} |b|
\eeq
for real numbers $a,b$, for $j\ge 1$, which can be checked easily by an induction and the trivial bound $|\re \left( (a+\ii b)^j\right)|\le |a+\ii b|^j$, 
and using the second part of Lemma~\ref{G gamma bound}, we find that 
\beq
\begin{split}
	|\Omega| 
	\leq \sum_{j=4}^\infty \frac{C_0^j N^{j-1+3\epsilon}}{N^{1-8\epsilon}} |X+\ii Y|^{j-1} 
	\leq 2C_0^4 N^{11\epsilon} \left( X^2+Y^2 \right) 
\end{split}
\eeq
for $z =(X+\gamma)+\ii Y\in \Gamma^+ \cap \overline{B}_{N^{-2}}$, for all $N> 2 C_0$.
From the second part of Lemma~\ref{G gamma bound}, 
\beq \label{eq:alphabd}
	\frac1{C_0}N^{1-15\epsilon}\le \alpha\le C_0^3N^{1+11\epsilon}, 
\eeq
and since $\epsilon<\frac1{26}$, we have 
\beq
	N^{11\epsilon} X^2 \ll \alpha X^2 \ll X, \qquad N^{11\epsilon} Y^2\ll \alpha Y^2
\eeq
uniformly for $z =(X+\gamma)+\ii Y\in \overline{B}_{N^{-2}}$
where the notation $a_N\ll b_N$ for sequences $a_N$ and $b_N$ means that $\frac{a_N}{b_N}\to 0$ as $N\to \infty$. 
Hence equation~\eqref{im on Gamma2} becomes
\beq
	X(1+o(1)) + \frac{\alpha}6 Y^2 (1+o(1))=0,
\eeq
and hence
\beq \label{eq: Gammaplus parabola}
	X = -\frac{\alpha}6  Y^2(1+ o(1)), 
	\qquad \alpha= - \frac{G'''(\gamma)}{G''(\gamma)} >0,
\eeq
uniformly for $z =(X+\gamma)+\ii Y\in \Gamma^+ \cap \overline{B}_{N^{-2}}$.
This shows, in particular, that $\re z < \gamma$ for $z\in \Gamma^+ \cap \overline{B}_{N^{-2}}$.
Moreover, it is direct to check, by proceeding as above, that 
\beq
	\partial_X F= 1 + o(1), 
	\qquad
	\partial_Y F = \frac{\alpha}3 Y (1 + o(1)), 	
\eeq
uniformly for $z =(X+\gamma)+\ii Y\in \overline{B}_{N^{-2}}$. 
(Here, for the estimate of $\partial_Y \Omega$, we use the inequality  
\beq
	|\re \left( (a+\ii b)^j - a^j\right)| \leq  \frac{j(j-1)}2 |a+\ii b|^{j-2} |b|^2
\eeq
for real numbers $a,b$, for $j\ge 2$, which can be checked easily by an induction and the bound~\eqref{eq:xplusyjima}.) 
Therefore, $\Gamma_+\cap \overline{B}_{N^{-2}}$ is a graph, and 
 $\dd y=\dd Y$ is positive on $\Gamma^+ \cap \overline{B}_{N^{-2}}$ and $\Gamma^+$ intersects $\partial B_{N^{-2}}$ at exactly one point. 

Let $z_2 \in \C^+$ be the point where $\Gamma^+$ and $\partial B_{N^{-2}}$ intersect.
Since $\Gamma_+$ is a path of steepest-descent, $\im G(z_2)=\im G(\gamma)=0$ and $G(z_2)<G(\gamma)$. 
From~\eqref{eq:G Taylor} and Lemma~\ref{G gamma bound}, 
\beq \label{contour outside large ball}
		G(\gamma)-G(z_2) = \left| G(z_2) - G(\gamma) \right|
	\ge \frac14 N^{1-8\epsilon} |z_2-\gamma|^2 - \sum_{j=3}^\infty \frac1{j} C_0^j N^{j-1+3\epsilon} |z_2-\gamma|^j
	\ge \frac18 N^{-3-8\epsilon}
\eeq
for all large enough $N$. 
On the other hand, consider $B_{N^{-3}}$, the ball of radius $N^{-3}$ centered at $\gamma$, and let $z_3 \in \C^+$ be the point where $\Gamma^+$ and $\partial B_{N^{-3}}$ intersect. Then, by a similar argument, 
\beq \label{contour inside small ball}
	G(\gamma) -G(z_3) = \left| G(z_3) - G(\gamma) \right| \le 2C_0^2 N^{-5+3\epsilon}
\eeq
for all large enough $N$ where $C_0$ is the constant in Lemma~\ref{G gamma bound}.

Next, we claim that, for any decreasing function $f: \Gamma^+ \to \R$, 
\beq \label{claim on decreasing function}
\int_{\Gamma^+} e^{f(z)} \dd y \geq 0.
\eeq
To check the claim, we parametrize the curve $\Gamma^+ = \Gamma^+ (t)$, $t \in [0, 1]$, with $\Gamma^+(0) = \gamma$ and $\Gamma^+(1) = -\infty$. 
Note that $\im \Gamma^+(0)=\im \Gamma^+(1)$. 
Suppose that $\im \Gamma^+ (t)$ increases on $(0, t_0)$ and decreases on $(t_0, 1)$. (In particular, it attains its maximum at $t=t_0$.) Then, we have
\beq \begin{split}
	\int_{\Gamma^+} e^{f(z)} \dd y 
	&= \int_0^1 e^{f(\Gamma^+(t))} \im (\Gamma^+)' (t) \dd t \\
	&= \int_0^{t_0} e^{f(\Gamma^+(t))} \im (\Gamma^+)' (t) \dd t + \int_{t_0}^1 e^{f(\Gamma^+(t))} \im (\Gamma^+)' (t) \dd t \\
	&\geq \int_0^{t_0} e^{f(\Gamma^+(t_0))} \im (\Gamma^+)' (t) \dd t + \int_{t_0}^1 e^{f(\Gamma^+(t_0))} \im (\Gamma^+)' (t) \dd t \\
	& =e^{f(\Gamma^+(t_0))} \int_0^1 \im (\Gamma^+)' (t) \dd t = 0.
\end{split} \eeq
This shows that the claim \eqref{claim on decreasing function} holds in this case. 
For the general case, suppose that there are real numbers 
$0<t_0<t_1<\cdots< t_{2k}<1$ such that $\im \Gamma^+ (t)$ increases on $(0, t_0), (t_1, t_2), \cdots, (t_{2k-1}, t_{2k})$ and decreases on $(t_0, t_1), (t_2, t_3), \cdots, (t_{2k}, 1)$. 
(From the explicit formula of $\im G(z)=0$, it is direct to check that $\Gamma^+$ approaches to $-\infty$ monotonically downward, and hence there are only finitely many such real numbers $t_i$'s.)
Set
$$
\Delta_{2p} = \im \Gamma^+(t_{2p}) - \im \Gamma^+(t_{2p-1}), \qquad \Delta_{2p+1} = \im \Gamma^+(t_{2p}) - \im \Gamma^+(2p+1), \qquad (p=0, 1, \cdots, k)
$$
with $t_{-1} := 0$ and $t_{2k+1} := 1$. Note that $\Delta_{2p}\geq 0$, $\Delta_{2p+1} \geq 0$, and
$$
\Delta_0 + \Delta_2 + \cdots + \Delta_{2p} \geq \Delta_1 + \Delta_3 + \cdots + \Delta_{2p+1}
$$
for any $p=0, 1, \cdots, k$ since $\Gamma^+$ is above the real axis. Applying the same strategy as in the case $k=0$, we obtain 
\beq \begin{split}
&\int_0^1 e^{f(\Gamma^+(t))} \im (\Gamma^+)' (t) \dd t \\
&\geq e^{f(\Gamma^+(t_0))} (\Delta_0 - \Delta_1) + e^{f(\Gamma^+(t_2))} (\Delta_2 - \Delta_3) + \cdots + e^{f(\Gamma^+(t_{2k}))} (\Delta_{2k} - \Delta_{2k+1}) \\
&\geq e^{f(\Gamma^+(t_2))} (\Delta_0 - \Delta_1 + \Delta_2 - \Delta_3) + e^{f(\Gamma^+(t_4))} (\Delta_4 - \Delta_5) +\cdots + e^{f(\Gamma^+(t_{2k}))} (\Delta_{2k} - \Delta_{2k+1}) \\
&\geq e^{f(\Gamma^+(t_{2k}))} (\Delta_0 - \Delta_1 + \Delta_2 - \Delta_3 + \cdots + \Delta_{2k} - \Delta_{2k+1}) \geq 0.
\end{split} \eeq
This proves the claim \eqref{claim on decreasing function}.

We apply~\eqref{claim on decreasing function} to the function 
$$
f(z) =
	\begin{cases}
	\frac{N}{2} \left( G(z_2) - G(\gamma) \right) & \text{ if } z \in \Gamma_+ \cap B_{N^{-2}}\\
	\frac{N}{2} \left( G(z) - G(\gamma) \right) & \text{ if } z \in \Gamma^+ \cap (B_{N^{-2}})^c, 
	\end{cases}
$$
and obtain 
\beq \label{K estimate 2}
\int_{\Gamma^+ \cap B_{N^{-2}}} \exp \left( \frac{N}{2} \left[ G(z_2) - G(\gamma) \right] \right) \dd y + \int_{\Gamma^+ \cap (B_{N^{-2}})^c} \exp \left( \frac{N}{2} \left[ G(z) - G(\gamma) \right] \right) \dd y \geq 0.
\eeq

Now, since $\dd y$ is positive in $\Gamma^+\cap B_{N^{-2}}$, we have the estimate (recall~\eqref{eq:define K})
\beq \begin{split} \label{K estimate 1}
 	\frac12K &= \int_{\Gamma^+} \exp \left( \frac{N}{2} \left[ G(z) - G(\gamma) \right] \right) \dd y \\
&= \int_{\Gamma^+ \cap B_{N^{-3}}} \exp \left( \frac{N}{2} \left[ G(z) - G(\gamma) \right] \right) \dd y + \int_{\Gamma^+ \cap (B_{N^{-3}})^c} \exp \left( \frac{N}{2} \left[ G(z) - G(\gamma) \right] \right) \dd y \\
&\geq \int_{\Gamma^+ \cap B_{N^{-3}}} \exp \left( \frac{N}{2} \left[ G(z_3) - G(\gamma) \right] \right) \dd y + \int_{\Gamma^+ \cap (B_{N^{-3}})^c} \exp \left( \frac{N}{2} \left[ G(z) - G(\gamma) \right] \right) \dd y.
\end{split} \eeq
Subtracting \eqref{K estimate 2} from \eqref{K estimate 1}, we find that
\beq \begin{split}
	\frac12K &\geq \int_{\Gamma^+ \cap B_{N^{-3}}} \left[ \exp \left( \frac{N}{2} \left[ G(z_3) - G(\gamma) \right] \right) - \exp \left( \frac{N}{2} \left[ G(z_2) - G(\gamma) \right] \right) \right] \dd y \\
	&\qquad + \int_{\Gamma^+ \cap (B_{N^{-3}})^c \cap B_{N^{-2}}} \left[ \exp \left( \frac{N}{2} \left[ G(z) - G(\gamma) \right] \right) - \exp \left( \frac{N}{2} \left[ G(z_2) - G(\gamma) \right] \right) \right] \dd y \\
	&\geq \int_{\Gamma^+ \cap B_{N^{-3}}} \left[ \exp \left( \frac{N}{2} \left[ G(z_3) - G(\gamma) \right] \right) - \exp \left( \frac{N}{2} \left[ G(z_2) - G(\gamma) \right] \right) \right] \dd y . 
\end{split} \eeq
Since $\int_{\Gamma^+ \cap B_{N^{-3}}} \dd y= \im z_3 \ge C|z_3-\gamma|= CN^{-3}$ for some constant $C>0$ (see~\eqref{eq: Gammaplus parabola} and~\eqref{eq:alphabd}), we find, 
using the estimates  \eqref{contour outside large ball} and \eqref{contour inside small ball}, that 
\beq \begin{split}
	\frac12 K &\geq CN^{-3} \left[ \exp \left( - C_0^2N^{-4+3\epsilon} \right) - \exp \left( - \frac1{16} N^{-2-8\epsilon} \right) \right] 
	\geq C N^{-5-8\epsilon}. 
\end{split} \eeq
This completes the proof of the lemma.
\end{proof}

We now prove Theorem \ref{thm:sup}.

\begin{proof}[Proof of Theorem \ref{thm:sup}]
From~\eqref{integral representation},~\eqref{n sphere}, and Lemma \ref{steepest descent sup}, we obtain
\beq
	Z_N = \frac{\sqrt{N}\beta}{\ii \sqrt{\pi} (2\beta e)^{N/2}} e^{\frac{N}{2} G(\gamma)} K (1+O(N^{-1}))
\eeq
with high probability, where $K$ satisfies $N^{-C}\le K\le C$. 
Thus, using Lemma~\ref{G gamma bound}, we find that
\beq \begin{split}
	F_N &= \frac{1}{N} \log Z_N = \frac{1}{2} \left[ G(\gamma)-1-\log(2\beta) \right]  + O(N^{-1} \log N) \\
	&= \beta \lambda_1  - \beta_c (\lambda_1-C_+) - \frac{1}{2}  \left( \int_{C_-}^{C_+} \log (C_+-z)d\nu(z)+ 1 +\log(2\beta) \right)+ O(N^{-1+4\epsilon}) \\
	&= \frac{1}{2} \left[  2\beta \lambda_1  - 2\beta_c (\lambda_1-C_+) -  \int_{C_-}^{C_+} \log (C_+-z)d\nu(z) -1 -\log(2\beta)\right] + O(N^{-1+4\epsilon})
\end{split} \eeq
with high probability.
Hence, we obtain that
\beq
	F_N - F(\beta) =  \left( \beta - \beta_c \right) (\lambda_1  -C_+)+ O(N^{-1 +4\epsilon}).
\eeq
where $F(\beta)$ is defined in~\eqref{eq:L sup}. 
Now the proof of the theorem follows from the edge universality, Condition \ref{cond:edge}. 
\end{proof}

\section{Third order phase transition} \label{sec:third}

In this section, we prove Theorem \ref{thm:third}.

\begin{proof}[Proof of Theorem \ref{thm:third}]
The convergence of $F_N$ to $F(\beta)$ in distribution follows from Theorem~\ref{thm:sub} and~\ref{thm:sup} since normal distribution and Tracy-Widom distribution have exponential tails. 

We now prove that $F(\beta)$ is $C^2$ but not $C^3$ at $\beta=\beta_c$. 
It suffices to prove the statement for 
\beq
\wt F(\beta) := \frac{1}{\beta} \left( F(\beta) + \frac{1 + \log (2\beta)}{2} \right).
\eeq
Set 
\beq \label{eq:f0beta}
f_0(\beta) = \int_{C_-}^{C_+} \log (\widehat\gamma - k) \dd \nu(k)
\eeq
as in~\eqref{eq:f_0}. 
Recall that $\widehat\gamma \equiv \widehat\gamma(\beta)$ is a function of $\beta$ and 
and it has a limit as $\beta\nearrow \beta_c$: 
$\widehat\gamma(\beta_c) = C_+$ (see the discussion after~\eqref{stieltjes gamma}).  
Hence, if we set 
\beq \label{eq:f0betac}
f_0(\beta_c) = \int_{C_-}^{C_+} \log (C_+ - k) \dd \nu(k),
\eeq
then $f_0(\beta)\to f_0(\beta_c)$ as $\beta \nearrow \beta_c$. 
With this definition,  
\beq
	\wt F(\beta) = 
	\begin{cases} \widehat\gamma(\beta) - \frac{1}{2\beta} f_0(\beta) , \qquad  & \beta<\beta_c,\\
	C_+ - \frac1{2\beta} f_0(\beta_c), \qquad & \beta\ge \beta_c.
	\end{cases}
\eeq
From this it easy to see that 
\beq\label{eq:diffsuplimits}
\lim_{\beta \searrow \beta_c} \partial_\beta \wt F(\beta) = \frac{f_0(\beta_c)}{2\beta_c^2}, \qquad \lim_{\beta \searrow \beta_c} \partial_\beta^2 \wt F(\beta) = -\frac{f_0(\beta_c)}{\beta_c^3}, \qquad \lim_{\beta \searrow \beta_c} \partial_\beta^3 \wt F(\beta) = \frac{3 f_0(\beta_c)}{\beta_c^4}.
\eeq

We now calculate the limits as $\beta\nearrow \beta_c$.
Since $\widehat\gamma(\beta)$ satisfies (see~\eqref{stieltjes gamma})
\beq \label{stieltjes gamma second}
	\int_{C_-}^{C_+} \frac{\dd \nu(k)}{\widehat\gamma(\beta) - k} = 2\beta,
\eeq
we find that
\beq
\partial_\beta f_0(\beta) = (\partial_\beta \widehat\gamma) \int_{C_-}^{C_+} \frac{\dd \nu(k)}{\widehat\gamma - k} = (2\beta) \partial_\beta \widehat\gamma.
\eeq
Thus,
\beq\label{eq:diff101}
\partial_\beta \wt F(\beta) = \partial_\beta \left( \widehat\gamma(\beta) - \frac{f_0(\beta)}{2\beta} \right)
= \frac{f_0(\beta)}{2\beta^2},
\eeq
and hence
\beq
\lim_{\beta \nearrow \beta_c} \partial_\beta \wt F(\beta) = \frac{f_0(\beta_c)}{2\beta_c^2} 
= \lim_{\beta \searrow \beta_c} \partial_\beta \wt F(\beta).
\eeq

Differentiating~\eqref{eq:diff101} again, we obtain
\beq\label{eq:diff102}
\partial_\beta^2 \wt F(\beta) = \frac{\partial_\beta \widehat\gamma}{\beta} - \frac{f_0(\beta)}{\beta^3}. 
\eeq
We claim that $\partial_\beta \widehat\gamma \to 0$ as $\beta \nearrow \beta_c$. 
This can be checked by differentiating \eqref{stieltjes gamma second}, 
\beq \label{stieltjes gamma derivative}
	-(\partial_\beta \widehat\gamma) \int_{C_-}^{C_+} \frac{\dd \nu(k)}{(\widehat\gamma - k)^2} = 2,
\eeq
and noting that the integral tends to $\infty$ as $\beta \nearrow \beta_c$ due to the square root decay of $\frac{\dd \nu}{dk}$.
Therefore, 
\beq
	\lim_{\beta \nearrow \beta_c} \partial_\beta^2 \wt F(\beta) = - \frac{f_0(\beta)}{\beta^3}
	= \lim_{\beta \searrow \beta_c} \partial_\beta^2 \wt F(\beta).
\eeq

Finally, after differentiating~\eqref{eq:diff102},
\beq
\partial_\beta^3 \wt F(\beta) = \frac{\partial_\beta^2 \widehat\gamma}{\beta} - \frac{\partial_\beta \widehat\gamma}{\beta^2} + \frac{3 f_0(\beta)}{\beta^4}. 
\eeq
Since $\partial_\beta \widehat\gamma \to 0$ as $\beta \nearrow \beta_c$, 
\beq
\lim_{\beta \nearrow \beta_c} \partial_\beta^3 \wt F(\beta) =  \frac{\lim_{\beta \nearrow \beta_c} \partial_\beta^2 \widehat\gamma}{\beta_c} + \frac{3 f_0(\beta_c)}{\beta_c^4}. 
\eeq
Comparing with~\eqref{eq:diffsuplimits}, it suffices to show that $\displaystyle\lim_{\beta \nearrow \beta_c} \partial_\beta^2 \widehat\gamma \neq 0$ to prove the theorem. 
Differentiating \eqref{stieltjes gamma derivative} once more, we obtain 
\beq \label{stieltjes gamma derivative 2}
	-(\partial_\beta^2 \widehat\gamma) \int_{C_-}^{C_+} \frac{\dd \nu(k)}{(\widehat\gamma - k)^2} + 2 (\partial_\beta \widehat\gamma)^2 \int_{C_-}^{C_+} \frac{\dd \nu(k)}{(\widehat\gamma - k)^3} = 0.
\eeq
Since $\frac{\dd \nu}{dk}$ has the square root decay, we have
\beq
	\frac2{\partial_\beta \widehat \gamma}= - \int_{C_-}^{C_+} \frac{\dd \nu(k)}{(\widehat\gamma - k)^2} \sim (\widehat\gamma - C_+)^{-1/2}, \qquad \int_{C_-}^{C_+} \frac{\dd \nu(k)}{(\widehat\gamma - k)^3} \sim (\widehat\gamma - C_+)^{-3/2},
\eeq
as $\beta\nearrow \beta_c$. 
We thus conclude from \eqref{stieltjes gamma derivative 2} that $\partial_\beta^2 \widehat\gamma  \sim 1$.
This completes the proof of the desired theorem.
\end{proof}

\begin{appendix}

\section{Appendix}

In this appendix, we evaluate various constants stated in Section~\ref{sec:examples}. 
As a main tool of the evaluation, we use the Stieltjes transform of the measure $\nu$, defined by
\beq
	m(z) = \int_{C_-}^{C_+} \frac{\dd \nu(x)}{x-z}, \qquad z \in \C\setminus [C_-, C_+).
\eeq
Note that due to the square root decay of $\frac{\dd \nu(x)}{dx}$ at $x=C_+$, $m(C_+)$ is well-defined. 

\subsection{Wigner matrix} \label{appendix Wigner}

We only consider real Wigner matrices. The complex case can be evaluated by the same way and we skip the details. 
The limiting spectral measure for real Wigner matrices is $\frac{\dd \nu(x)}{dx} = \frac1{2\pi} \sqrt{4-x^2}$, $-2\le x\le 2$. Hence 
\beq \label{stieltjes wigner}
	m(z) = \frac{-z+\sqrt{z^2 -4}}{2},
\eeq
and we find that 
\beq
	\beta_c = \frac{1}{2} \int_{C_-}^{C_+} \frac{\dd \nu(x)}{C_+-x} = -\frac{1}{2} m(2)= \frac12.
\eeq

We now evaluate $\widehat\gamma$ defined in~\eqref{stieltjes gamma}. 
This equation is equivalent to the equation $-m_{\nu}(\widehat\gamma) = 2\beta$. 
Solving this equation, we find that 
\beq \label{eq:gamma for Wigner real}
	\widehat\gamma = 2\beta + \frac{1}{2\beta}.
\eeq

To evaluate~\eqref{L sub010} and~\eqref{eq:L sup}, we note that 
\beq
	\int_{-2}^2 \log (z-x) \dd \nu(x) = \frac14 z \left( z - \sqrt{z^2 -4} \right)+ \log \left( z + \sqrt{z^2 -4} \right)-\log 2 - \frac{1}{2}. 
\eeq
for $z\in \C\setminus (-\infty, 2]$. This follows by noting that 
\beq \label{eq:Stiel integ}
 	-m(z)= \frac{d}{dz} \left[ \frac{1}4 z \left( z - \sqrt{z^2 -4} \right) + \log \left( z + \sqrt{z^2 -4} \right) \right]
\eeq
and $\int_{-2}^2 \log (z-x) \dd \nu(x)= \log z + O(z^{-1})$ as $z \to \infty$. 
Thus, using~\eqref{eq:gamma for Wigner real}, we obtain 
\beq
	F(\beta)=
	\begin{cases}
	\beta^2 & \text{ if } \beta < 1/2 \\
	2\beta - \frac{\log (2\beta) + 3/2}{2} & \text{ if } \beta > 1/2.
	\end{cases}
\eeq

\bigskip

We now evaluate $\ell(\beta)$ and $\sigma^2(\beta)$ for the case when $\beta<\beta_c$. 
First, since 
\beq
	\int_{C_-}^{C_+} \frac{\dd \nu(x)}{(z-x)^2} = m'(z) = -\frac12 + \frac{z}{2\sqrt{z^2-4}}, \qquad z\in \C\setminus [-2,2], 
\eeq
~\eqref{eq: ell one def} becomes 
\beq \label{eq:ell one Wig}
	\ell_1(\beta) 
	=\log (2\beta) - \frac{1}{2} \log \left( m'(\widehat\gamma(\beta)) \right) 
	=\frac12 \log (1-4\beta^2).
\eeq
Second, we evaluate (see~\eqref{Chebyshev formula}) $\tau_\ell(\varphi)= t_\ell(\widehat\gamma(\beta))$ for  $\ell=0,1,2,$ and $4$, 
where 
\beq \label{Chebyshev formula Wig2}
	t_\ell(z)= \frac1{\pi} \int_{-2}^2 \log(z-x) \frac{T_\ell(x/2)}{\sqrt{4-x^2}} dx, \qquad z\in \C\setminus(-\infty, 2),
\eeq
We start with evaluating 
\beq \label{Chebyshev formula Wig3}
	t_\ell'(z)= \frac1{\pi} \int_{-2}^2  \frac{T_\ell(x/2)}{(z-x)\sqrt{4-x^2}} dx. 
\eeq
Recall the recursions $T_{\ell+2}(y)=2yT_{\ell+1}(y)-T_{\ell}(y)$, $\ell\ge 0$, and the orthogonality relation $\int_{-1}^1  \frac{y^kT_\ell(y)}{\sqrt{1-y^2}} dy=0$ for all integers $0\le k<\ell$, for the Chebyshev polynomials. 
From this we obtain the recursions 
\beq \label{Chebyshev formula Wig4}
	t_{\ell+2}'(z)= z t_{\ell+1}'(z)-t_{\ell}'(z)
\eeq
for $\ell\ge 0$. Since $T_0(x/2)=1$ and $T_1(x/2)=x/2$, it is easy to check from residue calculus that
\beq \label{Chebyshev formula Wig5}
	t_{0}'(z) = \frac1{\sqrt{z^2-4}} , \qquad 
	t_{1}'(z) = \frac{z}{2\sqrt{z^2-4}} -\frac12. 
\eeq
Hence 
\beq \label{Chebyshev formula Wig6}
	t_{2}'(z)= \frac{z^2-2}{2\sqrt{z^2-4}} -\frac{z}2, 
	\qquad 
	t_4'(z)= \frac{z^4-4z^2+2}{2\sqrt{z^2-4}} -\frac{z^3}2+z. 
\eeq
Taking the anti-derivaties and noting from~\eqref{Chebyshev formula Wig2} that $\tau_0(z)=\log z+ O(1/z)$ and $t_\ell(z)=  O(1/z)$, $\ell\ge 1$, as $z\to \infty$ using the orthogonality relations of the Chebyshev polynomials, we find that 
\beq \label{Chebyshev formula Wig7}
	t_{0}(z)= \log\left( z+ \sqrt{z^2-4} \right) -\log 2, 
	\qquad 
	t_1(z)= \frac12 \sqrt{z^2-4}-\frac{z}2, 
\eeq
and
\beq \label{Chebyshev formula Wig8}
	t_2(z)= \frac{z}4 \sqrt{z^2-4}-\frac{z^2}4 + \frac12, 
	\qquad 
	t_4(z)= \frac{1}8 (z^3-2z)\sqrt{z^2-4}-\frac{z^4}8 + \frac{z^2}2 - \frac14.
\eeq
Hence,
\beq \label{Chebyshev formula Wig9}
	\tau_{0}(\varphi)= -\log (2\beta), 
	\quad 
	\tau_{1}(\varphi)= -2\beta, 
	\quad 
	\tau_{2}(\varphi)= -2\beta^2,
	\quad 
	\tau_{4}(\varphi)= -4\beta^4.
\eeq

Third, $V_{GOE}(\varphi)$ (see~\eqref{GOE variance}) is evaluated from the  lemma~\ref{eq:evaluation of int of F} below. 
Hence we find that
\beq
	V_{GOE}(\varphi)= \frac1{2\pi^2} L(\widehat\gamma(\beta), \widehat\gamma(\beta)) = -2 \log (1-4\beta^2)
\eeq
where $F(z,w)$ is defined in~\eqref{eq: F defined here}. 
We also have (see~\eqref{GOE mean})
\beq
	M_{GOE}(\varphi)= \frac12 \log (1-4\beta^2).
\eeq
Therefore, combining  with~\eqref{eq:ell one Wig},~\eqref{Chebyshev formula Wig9}, we obtain~\eqref{eq:ell Wigner} and~\eqref{eq:sigma Wigner}.

\bigskip

It remains to prove following Lemma. 

\begin{lem} \label{eq:evaluation of int of F}
Set
\beq \label{eq: F defined here}
	L(z,w)=\int_{-2}^2\int_{-2}^2 \left( \log(z-x)-\log(z-y)\right)\left( \log(w-x)-\log(w-y)\right)Q(x,y) \dd x \dd y, 
\eeq
for $z,w\in\C\setminus(-\infty, 2]$, 
where
\beq
	Q(x,y)= \frac{4-xy}{(x-y)^2\sqrt{4-x^2}\sqrt{4-y^2}}.
\eeq
Then
\beq
	L(z,w)=2\pi^2 \log \left[ \frac{(z+R(z))(w+R(w))}{2(zw-4+R(z)R(w))} \right] , \qquad R(z)=\sqrt{z^2-4},
\eeq
for $z,w\in\C\setminus(-\infty, 2]$, where $R(z)$ is defined with branch cut $[-2,2]$. 
\end{lem}

\begin{proof}[Proof of Lemma~\ref{eq:evaluation of int of F}]
Consider the second derivative 
\beq
	L_{zw}(z,w)=\int_{-2}^2 \left[ \int_{-2}^2\left( \frac1{z-x}-\frac1{z-y}\right)\left( \frac1{w-x}-\frac1{w-y}\right)\frac{4-xy}{(x-y)^2\sqrt{4-x^2}\sqrt{4-y^2}} \dd y \right] \dd x .
\eeq
Recall that $R(x)=\sqrt{x^2-4}$ is defined with the branch cut $[-2,2]$. Using this, it is easy to check that 
\beq
	L_{zw}(z,w)=-\frac14 \int_{\Sigma} \left[ \int_{\Sigma'} \left( \frac1{z-x}-\frac1{z-y}\right)\left( \frac1{w-x}-\frac1{w-y}\right)\frac{4-xy}{(x-y)^2R(x)R(y)} \dd y \right] \dd x 
\eeq
where $\Sigma$ and $\Sigma'$ are simple closed contours, oriented positively, that contain the interval $[-2,2]$ in its interior and the points $z$ and $w$ in its exterior, and $\Sigma'$ lies in the interior of $\Sigma$. 
By residue Calculus, we obtain
\beq
	L_{zw}(z,w)=\frac{-2\pi^2}{(z-w)^2} \left( 1+ \frac{4-zw}{R(z)R(w)} \right). 
\eeq
Integrating with respect to $z$, 
\beq \label{eq:Fw one}
	L_{w}(z,w)=\frac{-2\pi \left( R(z)-R(w)\right)}{(z-w)R(w)} +C_1(w)
\eeq
for some function $C_1(w)$ of $w$. 
The function $C_1(w)$ is determined by noting that, from~\eqref{eq:Fw one}, 
\beq
	L_{w}(z,w)=\frac{-2\pi }{R(w)} +C_1(w) + O(z^{-1})
\eeq
as $z\to\infty$ with fixed $w$, while from~\eqref{eq: F defined here}
\beq
	L_{w}(z,w)=L(z,w)=\int_{-2}^2\int_{-2}^2 \left( \log(z-x)-\log(z-y)\right)\left( \frac1{w-x}-\frac1{w-y}\right)Q(x,y) \dd x \dd y, 
\eeq
is $O(z^{-1})$. Hence $C_1(w)= \frac{2\pi^2}{R(w)}$. 
Integrating with respect to $w$, we find that 
\beq
	L(z,w)=2\pi^2 \log \left[ \frac{(w+R(w))}{2(zw-4+R(z)R(w))} \right] +C_2(z)
\eeq
for some function $C_2(z)$ of $z$. By considering the asymptotics as $w\to\infty$, we find that 
$C_2(z)=2\pi^2 \left[ \log(z+R(z))-\log w\right]$, and this completes the proof of Lemma. 
\end{proof}

\subsection{Sample covariance matrix} \label{appendix sample}

We again only consider real case here. Complex case is similar. 
The Stieltjes transform of the Marchenko-Pastur distribution $\dd \nu$~\eqref{eq:MP law} is given by
\beq
	m_{\nu}(z) = \frac{-z+\sqrt{C_+C_-}+\sqrt{(z-C_+)(z-C_-)}}{2z}, \qquad z\in \C\setminus [C_-, C_+]. 	
\eeq
Here $C_+=(\sqrt{d}+1)^2$ and $C_-=(\sqrt{d}-1)^2$. 
Hence, for real sample covariance matrices, we have
\beq
	\beta_c = \frac{1}{2} \int_{C_-}^{C_+} \frac{\dd \nu(x)}{C_+ -x} = -\frac{m_{\nu}(C_+)}{2}
	=  \frac{1}{2 (\sqrt d +1)}. 
\eeq

In the high temperature case, $\beta < \beta_c$, we get from the definition of $\widehat\gamma$ that 
\beq \label{eq:mnu an beta}
	-m_{\nu}(\widehat\gamma) = 2\beta.
\eeq
Solving the equation, we find that
\beq \label{eq:hamma for sco}
	\widehat\gamma = \frac{d}{1-2\beta}+ \frac{1}{2\beta}.
\eeq
We note that $\widehat\gamma$ is a decreasing function in $\beta$ and satisfies $\widehat\gamma>C+$ for $0<\beta< \beta_c$. 
In order to evaluate $L(\beta)$ in~\eqref{L sub010}, note that 
\beq \begin{split}
	&\frac{\dd}{\dd\beta} \int_{-2}^2 \log (\widehat\gamma-\lambda) \dd \nu(\lambda) 
	=\left( \frac{\dd \widehat\gamma}{\dd \beta} \right) \frac{\dd}{\dd\widehat\gamma} \int_{-2}^2 \log (\widehat\gamma-\lambda) \dd \nu(\lambda) \\
	&=  -m_{\nu}(\widehat\gamma) \left( \frac{\dd \widehat\gamma}{\dd \beta} \right)
	= 2\beta \left( \frac{2d}{(1-2\beta)^2} - \frac1{2\beta^2} \right)
\end{split} \eeq
using~\eqref{eq:mnu an beta} and~\eqref{eq:hamma for sco}. Hence, by taking the antiderivative, 
\beq \label{eq:f0 for sub sc}
\begin{split}
	\int_{-2}^2 \log (\widehat\gamma(\beta)-\lambda) \dd \nu(\lambda) 
&= \frac{2\beta }{1-2\beta}d + d \log (1-2\beta) - \log (2\beta),
\end{split} \eeq
where the constant of integration is determined by the asymptotics that the left hand side behaves as 
$$
\log \widehat\gamma + O(\widehat\gamma^{-1}) = -\log (2\beta) + O(\beta)
$$
as $\widehat\gamma \to \infty$, or $\beta \to 0$. 
Hence
$$
	L(\beta)=	- \frac{d}{2\beta} \log (1-2\beta) 
$$
for $\beta<\beta_c$. 

The limit of the free energy per particle $L(\beta)$ for $\beta>\beta_c$ is obtained easily using~\eqref{eq:f0 for sub sc} and noting that $C_+=\widehat\gamma(\beta_c)$.

We now evaluate $\ell$ and $\sigma^2$ for $\beta<\beta_c$. 
From~\eqref{eq: ell one def}, we find that 
\beq
	\ell_1 = \log(2\beta) - \frac1{2} \log (m'_{\nu}(\widehat\gamma))
	= \frac1{2} \log (1-4B^2)
\eeq
where we set
\beq
	B= \frac{\beta \sqrt{d}}{1-2\beta}. 
\eeq
On the other hand, in other to evaluate $M(\varphi)$ and $V(\varphi)$ for $\varphi(x)= \log(\widehat\gamma -x)$, 
we observe that~\eqref{eq:sc linear stat chang} implies that 
\beq
	\Phi(x)= \log(\sqrt{d}) +\psi(x), \qquad \psi(x)= \log \left( 2B+\frac1{2B}-x \right) 
\eeq
where $B$ is given above. 
Note that $\psi(x)$ is the same function as $\varphi(x)=\log(\widehat\gamma -x)$ for Wiger matrices with the change that $\beta$ is replaced by $B$. 
Moreover, we have $\tau_\ell(c+f)=\tau_{\ell}(f)$ for $\ell>0$ and $\tau_0(c+f)=c+\tau_0(f)$ for a constant $c$ and function $f$, from the definition of $\tau_\ell$ and the orthogonality of  Chebyshev polynomials. 
From this and the results obtained for Wigner matrices, we can easily find $M(\varphi)$ and $V(\varphi)$ for sample covariance matrices.

\end{appendix}


\def\cydot{\leavevmode\raise.4ex\hbox{.}}

\end{document}